\documentclass[final]{ws-ccm}

\usepackage{mathtools}

\newcommand{\ds}{\displaystyle}
\newcommand{\Nb}{{\mathbb{N}}}
\newcommand{\Rb}{{\mathbb{R}}}
\newcommand{\Zb}{{\mathbb{Z}}}
\newcommand{\A}{{\mathcal{A}}}

\newcommand{\C}{{\mathcal{C}}}

\newcommand{\F}{{\mathcal{F}}}

\newcommand{\K}{{\mathcal{K}}}
\newcommand{\LL}{{\mathcal{L}}}
\newcommand{\HH}{{\mathcal{H}}}
\newcommand{\M}{{\mathcal{M}}}

\newcommand{\med}{- \hskip -1em \int}

\newcommand{\res}{\mathop{\hbox{\vrule height 7pt width .5pt depth 0pt
\vrule height .5pt width 6pt depth 0pt}}\nolimits}

\def\hom{\text{hom}}
\def\dist{\text{dist}}

\let\e=\varepsilon
\let\O=\Omega
\let\o=\omega
\let\G=\Gamma
\let\a=\alpha
\let\b=\beta
\let\d=\delta

\let\wto=\rightharpoonup

\begin{document}

\markboth{J.-F. Babadjian \& V. Millot} {Homogenization of
variational problems in manifold valued Sobolev spaces}

\title{HOMOGENIZATION OF VARIATIONAL PROBLEMS\\ IN MANIFOLD VALUED SOBOLEV SPACES}
\author{Jean-Fran\c{c}ois BABADJIAN}

\address{Laboratoire Jean Kuntzmann\\
Universit\'e Joseph Fourier\\\
BP 53\\
38041 Grenoble Cedex 9, France\\
\emph{\tt{babadjia@imag.fr}}}

\author{Vincent MILLOT}

\address{D\'epartement de Math\'ematiques\\
Universit\'e de Cergy-Pontoise\\
Site de Saint Martin\\
2 avenue Adolphe Chauvin\\
95302 Cergy-Pontoise Cedex, France\\
\emph{\tt{vmillot@math.u-cergy.fr}}}

\maketitle

\begin{abstract}
{\bf Abstract.} Homogenization of integral functionals is studied
under the constraint that admissible maps have to take their values
into a given smooth manifold. The notion of tangential
homogenization is defined by analogy with the tangential
quasiconvexity introduced by Dacorogna, Fonseca, Mal\'y \& Trivisa
\cite{DFMT}. For energies with superlinear or linear growth, a
$\Gamma$-convergence result is established in Sobolev spaces, the
homogenization problem in the space of functions of bounded
variation being the object of \cite{BM}.
\end{abstract}





\section{Introduction}

\noindent The homogenization theory aims to find an effective
description of materials whose heterogeneities scale is much
smaller than the size of the body. The simplest example is
periodic homogenization for which the microstructure is assumed to
be periodically distributed within the material. In the framework of
the Calculus of Variations, periodic homogenization problems rest on
the study of equilibrium states, or minimizers, of integral
functionals of the form
\begin{equation}\label{mainfunct}\int_\O f\left(\frac{x}{\e},\nabla
u\right)dx\,,\quad u : \O \to \Rb^d\,,\end{equation} under suitable
boundary conditions, where $\O \subset \Rb^N$ is a bounded open set
and $f:\Rb^N \times \Rb^{d \times N} \to [0,+\infty)$ is some
oscillating integrand with respect to the first variable. To
understand the asymptotic behavior of (almost) minimizers of
such energies, it is convenient to perform a $\G$-convergence
analysis (see \cite{DM} for a detailed description of this subject)
which is an adequate theory to study such variational problems.
It is usual to assume that the integrand $f$ satisfies uniform
$p$-growth and $p$-coercivity conditions (with $1\leq p <+\infty$)
so that one should ask the admissible fields to belong to the Sobolev
space $W^{1,p}$.  For energies with  superlinear growth, {\it i.e.}, $p>1$,
this problem has a quite long history, and we refer to \cite{Mar}
in the convex case. Then it has received the most general answer in the independent works of
\cite{Br} and \cite{M}, showing that such materials asymptotically
behave like homogeneous ones. These results have been subsequently
generalized into a lot of different manners. Let us mention \cite{BDV} where the authors add a surface
energy term allowing for fractured media. In that case, Sobolev spaces
are not adapted to take into account eventual
discontinuities of the deformation field across the cracks.

In many applications admissible fields have to satisfy additional
constraints. This is for example the case in the study of equilibria
for liquid crystals, in ferromagnetism or for magnetostrictive
materials where the order parameters take their values into a given
manifold. It then becomes necessary to understand the
behaviour of integral functionals of the type (\ref{mainfunct})
under this additional constraint.
For fixed $\e>0$, the possible lack of lower semicontinuity of
the energy may prevent the existence of
 minimizers (with eventual boundary
conditions). It leads to compute its relaxation under the manifold
constraint. In the framework of Sobolev spaces, it has been studied
in \cite{DFMT,AL}, and the relaxed energy is obtained by replacing
the integrand by its tangential quasiconvexification which is the
analogue of the quasiconvex envelope in the non constrained case. We
finally mention a slightly different problem originally introduced
in \cite{BCL,BBC}, where the energy is assumed to be finite only for
smooth maps. Recent generalizations can be found in \cite{GM2} where
the study is performed within  the framework of Cartesian Currents
(see \cite{GMS}). It shows the emergence in the relaxation process
of non local effects of topological nature related to the non
density of smooth maps (see \cite{B,BZ}). \vskip5pt

The aim of this paper is to treat the problem of manifold
constrained homogenization, {\it i.e.}, the asymptotic as $\e\to0$
of energies of the form (\ref{mainfunct}) defined on manifold valued
Sobolev spaces. Let us make the idea more precise. We consider a
connected smooth submanifold $\M$ of $\Rb^d$ without
boundary. The tangent space of $\M$ at a point $s\in\M$ will be denoted by $T_s(\M)$.
The class of admissible maps we are interested in is
defined as
$$W^{1,p}(\O;\M):=\big\{ u \in W^{1,p}(\O;\Rb^d) : \; u(x) \in \M \text{ for $\LL^N$-a.e. }x \in \O\big\}\,.$$
For a smooth $\M$-valued map, it is well known that first order derivatives belong to the tangent  space of $\M$.
For $u\in W^{1,p}(\O;\M)$, this property still holds in the sense that $\nabla u(x)\in [T_{u(x)}(\M)]^N$ for $\LL^N$-a.e. $x\in\O$.
\vskip5pt

The energy density $f : \Rb^N \times \Rb^{d
\times N} \to [0,+\infty)$ is assumed to be a Carath\'eodory integrand satisfying
\begin{itemize}
\item[$(H_1)$] for every $\xi \in
\Rb^{d \times N}$ the function $f(\cdot,\xi)$ is $1$-periodic, {\it
i.e.}, if $\{e_1,\ldots,e_N\}$ denotes the canonical basis
of $\Rb^N$, one has $f(y+e_i,\xi)=f(y,\xi)$ for every $i=1,\ldots,N$ and $y \in \Rb^N$;\\
\item[$(H_2)$] there exist $0<\a \leq \b < +\infty$ and $1 \leq p<+\infty$
such that
$$\a |\xi|^p \leq f(y,\xi)\leq \b(1+|\xi|^p) \quad \text{ for a.e. }y \in \Rb^N \text{ and all }
\xi \in \Rb^{d \times N}\,.$$
\end{itemize}
For $\e>0$, we define the functionals $\F_\e:L^p(\O;\Rb^d) \to
[0,+\infty]$ by
$$\F_\e(u):=\begin{cases} \ds \int_\O f\left(\frac{x}{\e},\nabla
u\right) dx & \text{if }u \in
W^{1,p}(\O;\mathcal{M})\,,\\[8pt]
+\infty & \text{otherwise}\,.
\end{cases}$$
\vskip5pt

For energies with superlinear growth, we have the following result.

\begin{theorem}\label{babmil}
Let $\M$ be a connected smooth submanifold of $\Rb^d$ without
boundary, and $f:\Rb^N \times \Rb^{d \times N} \to [0,+\infty)$ be a
Carath\'eodory function satisfying $(H_1)$ and $(H_2)$ with
$1<p<+\infty$. Then the family $\{\F_\e\}_{\e>0}$ $\G$-converges for
the strong $L^p$-topology to the functional $\F_{\rm hom} :
L^p(\O;\Rb^d) \to [0,+\infty]$ defined by
$$\F_{\rm hom}(u):=\begin{cases} \ds \,\int_\O Tf_{\rm hom}(u,\nabla
u)\, dx & \text{if $u \in W^{1,p}(\O;\mathcal{M})$}\,,\\[8pt]
\,+\infty & \text{otherwise}\,,
\end{cases}$$
where for every $s\in \M$ and $\xi\in [T_s(\M)]^N$,
\begin{equation}\label{Tfhom}
Tf_{\rm hom}(s,\xi)=\lim_{t\to+\infty}\inf_{\varphi} \bigg\{
\med_{(0,t)^N} f(y,\xi+ \nabla \varphi(y))\, dy :  \varphi \in
W^{1,\infty}_0((0,t)^N;T_s(\mathcal{M})) \bigg\}
\end{equation}
is the tangentially homogenized energy density.
\end{theorem}

If the integrand $f$ has a linear growth in the $\xi$-variable, {\it
i.e.}, if $f$ satisfies $(H_2)$ with $p=1$, we assume in addition
that $\M$ is compact, and that
\begin{itemize}
\item[$(H_3)$]there exists $L>0$ such that
$$|f(y,\xi)-f(y,\xi')| \leq L |\xi-\xi'| \quad \text{ for a.e. }y \in \Rb^N \text{ and all }
\xi,\, \xi' \in \Rb^{d \times N}\,.$$
\end{itemize}

Then the following representation result on $W^{1,1}(\O;\M)$ holds:

\begin{theorem}\label{babmilp=1}
Let $\M$ be a connected and compact smooth submanifold of $\Rb^d$
without boundary, and  $f:\Rb^N \times \Rb^{d \times N} \to
[0,+\infty)$~be a Carath\'eodory function satisfying $(H_1)$ to
$(H_3)$ with $p=1$. Then the family $\{\F_\e\}_{\e>0}$
$\G$-converges for the strong $L^1$-topology at every $u \in
W^{1,1}(\O;\M)$ to $\F_{\rm hom} : W^{1,1}(\O;\M) \to [0,+\infty)$,
where
$$\F_{\rm hom}(u):= \int_\O Tf_{\rm hom}(u,\nabla u)\, dx\,,$$
and $Tf_{\rm hom}$ is given by (\ref{Tfhom}).
\end{theorem}

We would like to emphasize that the use of hypothesis $(H_3)$
is not too restrictive. Indeed,
 the $\G$-limit remains unchanged upon first relaxing the functional
$\F_\e$ (at fixed $\e>0$) in $W^{1,1}(\O;\Rb^d)$. It would
lead to replace the integrand $f$ by its tangential
quasiconvexification which, by virtue of the growth condition
$(H_1)$, does satisfy such a Lipschitz continuity assumption (see \cite{DFMT}).
\vskip5pt

We finally underline that Theorem \ref{babmilp=1} is not completely satisfactory in its present form. Indeed, in the case of an integrand with linear growth, the domain of the
$\G$-limit is obviously larger than the Sobolev space
$W^{1,1}(\O;\M)$ and the analysis has to be performed in the space of functions of bounded variation. In fact Theorem \ref{babmilp=1} is a
first step in this direction and the complete study in $BV$-spaces can be found in \cite{BM}.
\vskip5pt

The paper is organized as follows.
The study of the energy density $Tf_{\rm hom}$
and its main properties are presented in
Section 2. A locality property of the $\Gamma$-limit is established in Section 3. The upper bound inequalities
in  Theorems \ref{babmil} and \ref{babmilp=1} are the object of Section~4.  The lower bounds are obtained
 in Section 5 where  the proofs of both theorems are completed.

\subsection*{Notations}

\noindent We start by introducing some notations. Let $\O$ be a generic
bounded open subset of $\Rb^N$. We denote by  $\A(\O)$ the family of
all open subsets of $\O$.
We write  $B^k(s,r)$ for the closed ball in $\Rb^k$ of center $s \in
\Rb^k$ and radius $r>0$,
$Q:=(-1/2,1/2)^N$  the  open unit cube in $\Rb^N$, and
$Q(x_0,\rho):= x_0 + \rho \,Q$.
\vskip5pt

The space of real valued Radon measures in $\O$ with finite total
variation is denoted by $\M(\O)$.
We denote by $\LL^N$ the Lebesgue measure in
$\Rb^N$.
If $\mu \in \M(\O)$ and $\lambda \in \M(\O)$ is a nonnegative
Radon measure, we denote by $\frac{d\mu}{d\lambda}$ the
Radon-Nikod\'ym derivative of $\mu$ with respect to $\lambda$. By a
generalization of Besicovitch Differentiation Theorem (see
\cite[Proposition 2.2]{ADM}), there exists a Borel set $E$ such that $\lambda(E)=0$ and
$$\frac{d\mu}{d\lambda}(x)=\lim_{\rho \to 0^+} \frac{\mu(Q(x,\rho))}{\lambda(Q(x,\rho))}$$
for all $x \in {\rm Supp }\, \mu \setminus E$.

\section{Properties of the homogenized energy density}

In this section we present the main properties of the energy density $Tf_\hom$
defined in (\ref{Tfhom}).
 We  consider the bulk energy density
$$Tf_{\rm hom}(s,\xi) := \liminf_{t \to +\infty}\,\inf_{\varphi} \left\{ \med_{(0,t)^N}
f(y,\xi+ \nabla \varphi(y))\, dy : \varphi \in
W^{1,\infty}_0((0,t)^N;T_s(\mathcal{M})) \right\}$$ defined for
$s\in\M$ and $\xi\in [T_s(\M)]^N$. Our first concern is to
show that the $\liminf$ above is actually a limit. To this purpose
we shall introduce a new energy density $\bar f$ for which we
can apply  classical homogenization theories.
\vskip5pt

For $s \in \M$ we denote by $P_s:\Rb^d \to T_s(\M)$ the orthogonal
projection from $\Rb^d$ into $T_s(\M)$, and we set
$$\mathbf{P}_s(\xi):=(P_s(\xi_1),\ldots,P_s(\xi_N)) \quad \text{for
$\xi=(\xi_1,\ldots,\xi_N)\in\Rb^{d\times N}\,$.}$$
Given the Carath\'eodory
integrand $f:\Rb^N \times \Rb^{d\times N} \to [0,+\infty)$
satisfying assumptions $(H_1)$ and $(H_2)$ with $1 \leq p <+\infty$,
we define $\bar f:\Rb^N\times \M\times\Rb^{d\times N}\to
[0,+\infty)$ by
\begin{equation}\label{deffbar}
\bar f(y,s,\xi):=f(y,\mathbf{P}_s(\xi))+|\xi-\mathbf{P}_s(\xi)|^p\,.
\end{equation}
The new integrand $\bar f$ is a Carath\'eodory function,
and $\bar f(\cdot,s,\xi)$ is $1$-periodic for every
$(s,\xi)\in\M\times\Rb^{d\times N}$. By assumption $(H_2)$,
$\bar f$ also satisfies uniform $p$-growth and $p$-coercivity
conditions, {\it i.e.},
\begin{equation}\label{growthfbar}
\alpha'|\xi|^p\leq \bar f(y,s,\xi)\leq
\beta'(1+|\xi|^p)\quad\text{for every
}(s,\xi)\in\M\times\Rb^{d\times N} \text{ and a.e. }y\in \Rb^N\,,
\end{equation}
for some constants $0<\alpha'\leq \beta'<+\infty$.

\begin{proposition}\label{properties1}
Let $f:\Rb^N \times \Rb^{d\times N} \to [0,+\infty)$ be a
Carath\'eodory integrand satisfying $(H_1)$ and $(H_2)$ with $1 \leq
p <+\infty$. Then the following properties hold:
\begin{itemize}
\item[(i)] for every $s \in \M$ and $\xi \in [T_s(\M)]^N$,
$$Tf_{\rm hom}(s,\xi)= \lim_{t \to +\infty} \inf_{\varphi} \left\{
\med_{(0,t)^N} f(y,\xi+ \nabla \varphi(y))\, dy : \varphi \in
W^{1,\infty}_0((0,t)^N;T_s(\mathcal{M})) \right\}\,,$$ and
\begin{equation}\label{identhomform}
Tf_{\rm hom}(s,\xi)=\bar f_{\rm hom}(s,\xi)\,,
\end{equation}
where
$$\bar f_{\rm hom}(s,\xi):=\lim_{t\to+\infty}\inf_{\varphi}\bigg\{
\med_{(0,t)^N} \bar f(y,s,\xi+\nabla \varphi(y))\, dy: \varphi
\in W^{1,\infty}_0((0,t)^N;\Rb^d) \bigg\}$$ is the usual homogenized
energy density of $\bar f$ (see, e.g.,
\cite[Chapter~14]{BD});\\
\item[(ii)] the function $Tf_{\rm hom}$ is tangentially
quasiconvex, {\it i.e.}, for all $s \in \M$ and all $\xi \in
[T_s(\M)]^N$,
$$Tf_{\rm hom}(s,\xi) \leq \int_Q Tf_{\rm hom}(s,\xi + \nabla \varphi(y))\, dy$$
for every $\varphi \in W^{1,\infty}_0(Q;T_s(\M))$. In particular
$Tf_{\rm hom}(s,\cdot)$ is rank one convex;\\
\item[(iii)] there exists
$C>0$ such that
\begin{equation}\label{pgT}
\a|\xi|^p \leq Tf_{\rm hom}(s,\xi) \leq \b(1+|\xi|^p)\,,
\end{equation}
and
\begin{equation}\label{plipT}
|Tf_{\rm hom}(s,\xi)-Tf_{\rm hom}(s,\xi')| \leq
C(1+|\xi|^{p-1}+|\xi'|^{p-1})|\xi-\xi'|
\end{equation}
for every $s \in \M$ and $\xi$, $\xi' \in [T_s(\M)]^N$.
\end{itemize}
\end{proposition}

\begin{proof}
Fix $s \in \M$ and $\xi \in [T_s(\M)]^N$. For any $t>0$, we
introduce
$$Tf_{t}(s,\xi) :=\inf_{\varphi} \bigg\{ \med_{(0,t)^N}
f(y,\xi+ \nabla \varphi)\, dy :  \varphi \in
W^{1,\infty}_0((0,t)^N;T_s(\mathcal{M})) \bigg\} \,,$$
and
$$\bar f_{t}(s,\xi):=\inf_{\varphi}\bigg\{ \med_{(0,t)^N} \bar f(y,s,\xi+\nabla \varphi)\, dy:
\varphi \in W^{1,\infty}_0((0,t)^N;\Rb^d) \bigg\}\,.$$
By classical
results (see, {\it e.g.}, \cite[Proposition~14.4]{BD}), there exists
$$\ds\lim_{t\to+\infty} \bar f_t(s,\xi)\quad\text{for every $s\in\M$ and
$\xi\in [T_s(\M)]^N$}\,.$$
Hence to prove \emph{(i)}, it suffices to show
that $Tf_t(s,\xi)=\bar f_t(s,\xi)$ for every $t>0$.  For any
$\varphi \in W^{1,\infty}_0((0,t)^N;T_s(\M))$, we have
$$
\bar f_t(s,\xi)  \leq \med_{(0,t)^N}\bar f(y,s,\xi+\nabla
\varphi)\, dy = \med_{(0,t)^N}  f(y,\xi+\nabla \varphi)\, dy\,,
$$
since $\xi+\nabla \varphi(y)\in [T_s(\M)]^N$ for a.e. $y\in(0,t)^N$.
Taking the infimum over all such $\varphi$'s in the right hand side of
the previous inequality yields $\bar f_t(s,\xi) \leq
Tf_t(s,\xi)$. To prove the converse inequality we pick up $\psi \in
W^{1,\infty}_0((0,t)^N;\Rb^d)$ and we set $\tilde\psi=P_s(\psi)$.
One easily checks that $\tilde\psi \in
W^{1,\infty}_0((0,t)^N;T_s(\M))$ and
$\nabla\tilde\psi=\mathbf{P}_s(\nabla\psi)$ a.e. in $(0,t)^N$.
Therefore
$$Tf_t(s,\xi)  \leq  \med_{(0,t)^N}
f(y,\xi+\nabla \tilde\psi)\, dy = \med_{(0,t)^N}
f\big(y,\mathbf{P}_s(\xi+\nabla \psi)\big)\, dy \leq \med_{(0,t)^N}
\bar f(y,s,\xi+\nabla \psi)\, dy\,.$$ Then the converse
inequality arises taking the infimum over all admissible $\psi$'s.

By standard results $\bar f_\hom(s,\cdot)$ is a quasiconvex
function for every $s\in\M$ (see, {\it e.g.},
\cite[Theorem~14.5]{BD}). As a consequence, for any $s\in \M$,  $\xi
\in [T_s(\M)]^N$ and $\varphi \in W^{1,\infty}_0(Q;T_s(\M))$, we
have
$$Tf_\hom(s,\xi)=\bar f_\hom(s,\xi)  \leq  \int_Q \overline
f_\hom(s,\xi+\nabla \varphi)\, dy =  \int_Q Tf_\hom(s,\xi+\nabla
\varphi)\, dy\,,$$ which proves that $Tf_\hom$ is tangentially
quasiconvex. As a consequence of (\ref{identhomform}) and the
fact that $\bar f_\hom(s,\cdot)$ is rank one convex, it follows
that $Tf_\hom(s,\cdot)$ is rank one convex as well.

The proof of (\ref{pgT}) is immediate in view of $(H_1)$ and
the definition of $Tf_\hom$. Moreover rank one convex functions satisfying uniform $p$-growth and
$p$-coercivity conditions are $p$-Lipschitz (see, {\it e.g.},
\cite[Lemma~2.2,~Chap.~4]{D}), and thus  (\ref{plipT}) holds.
\end{proof}

\begin{remark}
It readily follows from the previous proof that Proposition
\ref{properties1} still holds for any Carath\'eodory integrand $\hat
f:\Rb^N\times\M \times \Rb^{d\times N} \to [0,+\infty)$ instead of
$\bar f$, provided that: $\hat f(x,s,\xi)=f(y,\xi)$ for every
$s\in\M$, every $\xi\in[T_s(\M)]^N$ and a.e. $y\in\Rb^N$; $\hat
f(\cdot,s,\cdot)$ satisfies $(H_1)$ and $(H_2)$ for every $s\in\M$
with uniform estimates with respect to $s$.
 \end{remark}

\begin{remark}\label{rem1d}
If ${\rm dim}(\M)=1$ then $T_s(\M)$ is a one dimensional linear
subspace of $\Rb^d$ for every $s \in \M$. Hence, given $s \in \M$, we can identify
$T_s(\M)$ with $\Rb$ through some linear mapping $i_s:\Rb\to T_s(\M)$. Using the application $i_s$,
we can also identify $[T_s(\M)]^N$ with $\Rb^N$ setting for  $z=(z_1,\ldots,z_N)\in\Rb^N$, $i_s(z):=(i_s(z_1),\ldots,i_s(z_N))$.
Define
$\hat f(y,s,z):=f(y,i_s(z))$ for $(y,s,z)\in\O\times\M\times\Rb^N$. By (\ref{identhomform}) and
\cite[Remark~14.6]{BD}, we can replace  in formula (\ref{Tfhom}) homogeneous boundary conditions
by periodic boundary conditions, and the limit as $t\to+\infty$ by the infimum over all $t\in\Nb$.
Moreover, in the scalar case the homogenization formula can be reduced to a single cell formula (see, {\it e.g.}, \cite[Chapter 14]{BD}).
Therefore
\begin{align*}
Tf_\hom(s,\xi) & =  \inf_{t \in \Nb} \,\inf \left\{\med_{(0,t)^N}
f(y,\xi+\nabla \varphi)\, dy : \varphi \in
W^{1,\infty}_\#((0,t)^N;T_s(\M)) \right\}\\
& = \inf_{t \in \Nb}\, \inf \left\{\med_{(0,t)^N} \hat f(y,s,i_s^{-1}(\xi)+\nabla
\phi)\, dy : \phi \in
W^{1,\infty}_\#((0,t)^N) \right\}\\
& =  \inf \left\{\int_{Q} \hat f(y,s,i_s^{-1}(\xi)+\nabla \phi)\, dy :
\phi \in W^{1,\infty}_\#(Q) \right\}\\
&=\inf\left\{ \int_Q f(y,\xi+ \nabla \varphi)\, dy : \varphi \in
W^{1,\infty}_\#(Q;T_s(\mathcal{M})) \right\}\,.
\end{align*}
This remark states that whenever the manifold $\M$ is one dimensional,
test functions in the minimization problem (\ref{Tfhom}) are in fact
scalar valued, and thus, one can compute the tangentially homogenized
energy density over one single cell instead of an infinite
set of cells. Note that this is not true in general even in the
non constrained case (see, {\it e.g.}, the counter-example in
\cite[Theorem~4.3]{M}).
\end{remark}
\vskip5pt

We conclude this section with an elementary example where the dependence on the $s$-variable is explicit.
It shows that \emph{tangential homogenization} does not reduce in general to standard homogenization. The construction is
based on a rank one laminate for which direct computations can be performed.

\begin{example}\label{exemple}
Assume that $\M=\mathbb{S}^1$ and for $x\in \Rb^N$, $\xi=(\xi_{ij})\in\Rb^{2\times N}$,
$$f(x,\xi)=\sum_{j=1}^N\big(a(x_1)|\xi_{1j}|^2+b(x_1)|\xi_{2j}|^2\big)\,, $$
where $a,b\in L^{\infty}(\Rb)$ are 1-periodic and bounded from below by a positive constant.
Arguing as in Remark~\ref{rem1d} and \cite[Example 25.6]{DM}, one may compute for $s=(s_1,s_2)\in\mathbb{S}^1$ and $\xi\in[T_s(\mathbb{S}^1)]^N$,
$$Tf_{\rm hom}(s,\xi)=\sum_{j=1}^N\alpha_j(s)\big(|\xi_{1j}|^2+|\xi_{2j}|^2\big)\,, $$
with
$$\alpha_j(s)=\begin{cases}
\ds \bigg(\int_{-1/2}^{1/2}\frac{dt}{a(t)s_2^2+b(t)s_1^2}\bigg)^{-1}&\text{ if $j=1$}\,,\\[10pt]
\ds \int_{-1/2}^{1/2}\big(a(t)s_2^2+b(t)s_1^2\big)\,dt & \text{otherwise}\,.
\end{cases}$$
Compare this result with \cite[Example 25.6]{DM}.
\end{example}

To treat the homogenization problem with $p=1$, we will need to extend the function $\bar f$ to the whole space $\Rb^N \times \Rb^d \times \Rb^{d \times
N}$.
We state in the following lemma our extension procedure.

\begin{lemma}\label{defg}
Assume that $\M$ is compact. Let $f:\Rb^N \times \Rb^{d\times N} \to [0,+\infty)$ be a
Carath\'eodory function satisfying $(H_1)$ to $(H_3)$ with $p=1$. Then there exists a Carath\'eodory function
$g:\Rb^N\times\Rb^d\times\Rb^{d\times N}\to[0,+\infty)$ such that
\begin{equation}\label{idfg}
g(y,s,\xi)=f(y,\xi)
\quad \text{for $s
\in \M$ and $\xi \in [T_s(\M)]^N\,$,}
\end{equation}
and satisfying :
\begin{itemize}
\item[(i)] $g$ is $1$-periodic in the first variable;
\item[(ii)] there exist $0<\alpha'\leq \b'$ such that
\begin{equation}\label{pgrowth}
\alpha'|\xi|\leq g(y,s,\xi)\leq
\b'(1+|\xi|)
\end{equation}
for every $(s,\xi)\in \Rb^d \times \Rb^{d
\times N}$ and a.e. $y \in \Rb^N$\,;
\item[(iii)] there exist $C>0$ and $C'>0$ such that
\begin{equation}\label{moduluscont}
|g(y,s,\xi)-g(y,s',\xi)| \leq C|s-s'|\; |\xi|\,,\end{equation} and
\begin{equation}\label{lipg}
|g(y,s,\xi) - g(y,s,\xi')| \leq C'|\xi-\xi'|\end{equation} for every
$s$, $s' \in \Rb^d$, every $\xi \in \Rb^{d \times N}$ and a.e. $y
\in \Rb^N$.
\end{itemize}
\end{lemma}

\begin{proof} For $\delta_0>0$ fixed, let  $\mathcal U:=\big\{s
\in\Rb^d\,:\,\dist(s,\M)<\d_0\big\}$ be the $\d_0$-neighborhood of
$\M$. Choosing $\delta_0>0$ small enough, we may assume that the
nearest point projection $\Pi: \mathcal U \to \M$ is a well defined
Lipschitz mapping. Then the map $s \in \mathcal U \mapsto
P_{\Pi(s)}$ is Lipschitz. Now we introduce a cut-off function $\chi
\in \C^\infty_c(\Rb^d;[0,1])$ such that $\chi(t)=1$ if $\dist(s,\M)
\leq \delta_0/2$, and $\chi(s)=0$ if $\dist(s,\M) \geq 3\delta_0/4$.
We define
$$\mathbb{P}_s(\xi):=\chi(s) \mathbf{P}_{\Pi(s)}(\xi)\quad \text{for $(s,\xi) \in \Rb^d \times \Rb^{d \times N}\,$.}$$
We consider the integrand $g:
\Rb^N\times\Rb^d\times\Rb^{d\times N}\to[0,+\infty)$ given by
\begin{equation*}
g(y,s,\xi)=f(y,\mathbb{P}_s(\xi))+|\xi-\mathbb{P}_s(\xi)|\,.
\end{equation*}
One may check that $g$ is  a Carath\'eodory
function, that $g(\cdot,s,\xi)$ is $1$-periodic for every $(s,\xi)
\in \Rb^d \times \Rb^{d \times N}$,  and that  $(H_2)$ yields \eqref{pgrowth}.
Then \eqref{moduluscont} and \eqref{lipg} follow from $(H_3)$ and the Lipschitz continuity of $s \mapsto \mathbb P_s$.
\end{proof}

\begin{remark}\label{idfhomghom}
In view of \eqref{idfg}, one may argue exactly as in
the proof of (\ref{identhomform}) to show that
\begin{equation}\label{f=g}
Tf_\hom(s,\xi)=g_\hom(s,\xi)\quad \text{for every $s
\in \M$ and $\xi \in [T_s(\M)]^N\,$,}
\end{equation}
where
$$g_\hom(s,\xi):=\lim_{t \to +\infty}\inf_{\varphi}\left\{\med_{(0,t)^N} g(y,s,\xi+\nabla \varphi(y))\, dy:\, \varphi
\in W^{1,\infty}_0((0,t)^N;\Rb^d) \right\}\,.$$
Hence upon extending
$Tf_\hom$ by $g_\hom$ outside the set $\big\{(s,\xi) \in \Rb^d \times
\Rb^{d \times N}:\; s \in \M,\, \xi \in [T_s(\M)]^N\big\}$, we can
tacitly assume $Tf_\hom$ to be defined over the whole $\Rb^d \times
\Rb^{d \times N}$.
\end{remark}

\section{Localization}

\noindent
In this section we show that a suitable functional larger than the $\G$-limit is a measure.
It will allow us to obtain the upper bound on the $\Gamma$-limit (see Lemma~\ref{upperbound}) through the blow-up method introduced in \cite{FM,FM2}.
\vskip5pt

Let us consider an arbitrary sequence $\{\e_n\} \searrow 0^+$. Along this sequence we define the
$\G(L^p)$-lower limit $\F:L^p(\O;\Rb^d)\to[0,+\infty]$ by
$$\F(u):= \inf_{\{u_n\}} \left\{ \liminf_{n \to +\infty}\, \F_{\e_n}
(u_n) \,: \,u_n\in W^{1,p}(\O;\M)\,, \,u_n \to u \text{ in }L^p(\O;\Rb^d) \right\}\,.$$
The idea is to localize the functionals  $\{\F_{\e_n}\}_{n\in\Nb}$ on the family $\A(\O)$ of all open
subsets of $\O$. For every $u \in L^p(\O;\Rb^d)$
and every $A \in \A(\O)$, define
$$\F_{\e_n}(u,A):\begin{cases}
\ds \int_A
f\left(\frac{x}{\e_n},\nabla u\right) dx & \text{if }u \in
W^{1,p}(\O;\M)\,,\\[8pt]
+\infty & \text{otherwise}\,.
\end{cases}
$$
Given a compact set $\mathcal{K}\subset \M$ and a subsequence
$\{\e_k\}:= \{\e_{n_k}\} \searrow 0^+$, we introduce for $u \in
W^{1,p}(\O;\M)$ and $A \in \A(\O)$,
\begin{multline*}
\F^{\{\e_{k}\}}_{\mathcal{K}}(u,A) :=  \inf_{\{u_k\}} \Bigg\{ \limsup_{k \to +\infty}\,
\F_{\e_{k}} (u_k,A) :  \, u_k \wto u \text{ weakly in }W^{1,p}(\O;\Rb^d)\,, \\
u_k \to u \text{ uniformly and }
u_k(x)=u(x) \text{ whenever } {\rm dist}\,(u(x),\K)>1\text{ for a.e. } x\in \O\,
\Bigg\}\,.
\end{multline*}
A key point in the upcoming analysis is the following locality result.

\begin{lemma}\label{meas}
For every $u \in W^{1,p}(\O;\M)$, there exists a subsequence $\{\e_{k}\}$ such that  the set function
$\F^{\{\e_{k}\}}_{\K}(u,\cdot)$ is the restriction to $\A(\O)$ of a Radon
measure absolutely continuous with respect to the Lebesgue measure
$\LL^N$.
\end{lemma}

\begin{proof}
>From the $p$-growth condition $(H_2)$ we infer that for any
subsequence $\{\e_{k}\}$,
\begin{equation}\label{absp}\F^{\{\e_{k}\}}_{\K}(u,A) \leq \b\int_A (1+|\nabla
u|^p)\, dx\,,\end{equation}
so it remains to prove the existence of a suitable subsequence $\{\e_{k}\}$ for which
$\F^{\{\e_{k}\}}_{\K}(u,\cdot)$ is (the trace of) a Radon measure.
\vskip5pt

{\bf Step 1.} We start by proving that for any subsequence $\{\e_{k}\}$ the following subadditivity
property holds:
\begin{equation}\label{sub} \F^{\{\e_{k}\}}_{\K}(u,A) \leq
\F^{\{\e_{k}\}}_{\K}(u,B) + \F^{\{\e_{k}\}}_{\K}(u,A \setminus \overline C)\,
\end{equation} for every $A$, $B$ and $C \in \A(\O)$ such that $\overline C
\subset B \subset A$.
Given $\eta >0$ arbitrary, there
exist sequences $\{u_k\}$, $\{v_k\} \subset W^{1,p}(\O;\M)$
such that $u_k$ and $v_k$ converge weakly to
$u$ in $W^{1,p}(\O;\Rb^d)$, $u_k(x)=v_k(x)=u(x)$ if ${\rm dist}\,(u(x),\K) >1$ for a.e. $x\in\O$,  $u_k$ and $v_k$ are uniformly
converging to $u$, and
\begin{equation}\label{u_n}
\left\{
\begin{array}{l}
\ds \limsup_{k \to +\infty}\, \F_{\e_{k}}(u_k,B) \leq \F^{\{\e_{k}\}}_{\K}(u,B) +
\eta\,,\\[0.3cm]
\ds \limsup_{k \to +\infty}\, \F_{\e_{k}}(v_k,A \setminus \overline C)
\leq \F^{\{\e_{k}\}}_{\K}(u,A \setminus \overline C) + \eta\,.
\end{array}\right.
\end{equation}
Let $\K':=\big\{s\in\M\,:{\rm dist}\,(s,\K)\leq 1\big\}$, then $\K'$ is a compact subset of $\M$ and
$u_k(x)=v_k(x)=u(x)$ if $u(x) \not\in \K'$ for a.e. $x\in\O$.

Consider $L:=\dist(C,\partial B)$, $M \in \Nb$, and for every $i \in
\{0,\ldots,M\}$ define
$$B_i:=\left\{x \in B:\, \dist(x,\partial B)
>\frac{iL}{M}\right\}\,.$$
Given $i \in \{0,\ldots,M-1\}$ let $S_i:=B_i \setminus
\overline{B_{i+1}}$, and $\zeta_i \in \C^\infty_c(\O;[0,1])$ be a
cut-off function satisfying
$$
\zeta_i(x)= \begin{cases}
1 & \text{in }  B_{i+1}\,,\\
0 & \text{in }  \O \setminus B_i\,,
\end{cases}
\text{ and}\quad |\nabla \zeta_i|\leq \frac{2\, M}{L}\,.
$$
By Lemma 3.2 and Remark 3.3 in \cite{DFMT}, there exist $\d>0$,
$c>0$, and a uniformly continuously differentiable mapping $\Phi:
D_\d \times [0,1] \to \M$, where
$$D_\d:=\big\{(s_0,s_1) \in \M \times \M : \dist(s_0,\K') < \d, \;
\dist(s_1,\K') < \d, \; |s_0-s_1|< \d\big\}\,,$$ such that
\begin{equation}\label{Phi}
\Phi(s_0,s_1,0)=s_0\,, \quad \Phi(s_0,s_1,1)=s_1\,,\quad \frac{\partial
\Phi}{\partial t}(s_0,s_1,t) \leq c|s_0-s_1|\,, \end{equation} and
\begin{equation}\label{lip}
|\Phi(s_0,s_1,t)-s_0| \leq c|s_0-s_1|\,.
\end{equation}
Since $\{u_k\}$ and $\{v_k\}$ are uniformly converging to $u$, one
can choose $k$ large enough to ensure that
$$\|u_k - u\|_{L^\infty(\O;\Rb^d)} < \d\,, \quad \|v_k - u\|_{L^\infty(\O;\Rb^d)} < \d \quad \text{ and }\quad \|u_k-
v_k\|_{L^\infty(\O;\Rb^d)} < \d\,.$$ Therefore for a.e.
$x\in\O$, $\dist(u_k(x),\K') < \d$ and $\dist(v_k(x),
\K') < \d$ whenever $u(x) \in \K'\,$.
Now we
are allowed to define
$$w_{k,i}(x):=\begin{cases}
\Phi(v_k(x),u_k(x),\zeta_i(x)) & \text{ if }  u(x) \in \K'\,,\\
u(x) & \text{ if }  u(x) \not\in\K'\,,\end{cases}$$
and $w_{k,i} \in W^{1,p}(\O;\M)$. Using the $p$-growth condition
$(H_2)$ together with (\ref{Phi}), we derive
\begin{multline*}
\int_A f \left(\frac{x}{\e_k},\nabla w_{k,i}\right) dx  \leq
\int_B f \left(\frac{x}{\e_k},\nabla u_k\right) dx +
\int_{A\setminus \overline C} f \left(\frac{x}{\e_k},\nabla
v_k\right)
dx\,+\\
+C_0 \int_{S_i}( 1 + |\nabla u_k|^p + |\nabla v_k|^p + M^p |u_k -
v_k|^p)\, dx\,,
\end{multline*}
for some constant $C_0>0$ independent of $k$, $i$ and $M$. Summing up over $i \in \{0,\ldots,M-1\}$ and dividing by $M$ yields
\begin{multline*}
\frac{1}{M}\sum_{i=0}^{M-1}\int_A f \left(\frac{x}{\e_k},\nabla
w_{k,i}\right) dx  \leq \int_B f \left(\frac{x}{\e_k},\nabla
u_k\right) dx + \int_{A\setminus \overline C} f
\left(\frac{x}{\e_k},\nabla v_k\right)
dx\,+\\
+\frac{C_0}{M} \int_{B \setminus \overline C}( 1 + |\nabla u_k|^p +
|\nabla v_k|^p + M^p |u_k - v_k|^p)\, dx\,.
\end{multline*}
Hence one may find some $i_k \in \{0,\ldots,M-1\}$ such that $\bar w_k:=w_{k,i_k}$ satisifies
\begin{multline}\label{1006}
\int_A f \left(\frac{x}{\e_k},\nabla \bar w_{k}\right) dx  \leq
\int_B f \left(\frac{x}{\e_k},\nabla u_k\right) dx +
\int_{A\setminus \overline C} f \left(\frac{x}{\e_k},\nabla
v_k\right)
dx\,+\\
+\frac{C_0}{M} \int_{B \setminus \overline C}( 1 + |\nabla u_k|^p +
|\nabla v_k|^p + M^p |u_k - v_k|^p)\, dx\,.
\end{multline}
>From (\ref{Phi}) and (\ref{lip}) we deduce that $\bar w_{k} \to u$
uniformly, $\bar w_{k} \wto u$ in $W^{1,p}(\O;\Rb^d)$, and $\bar
w_k(x)=u(x)$ if ${\rm dist}(u(x),\K)>1$ for a.e. $x\in\O$. Taking
$\{\bar w_{k}\}$ as competitor for $\F^{\{\e_k\}}_\K(u,A)$, and
using (\ref{1006}) together with (\ref{u_n}) leads to
\begin{align*}
\F^{\{\e_k\}}_\K(u,A) &\leq  \limsup_{k \to +\infty}\,\F_{\e_k}(\bar w_{k},A)\\
&\leq
\begin{multlined}[t][10cm]
\limsup_{k\to +\infty}\Big\{
\F_{\e_{k}}(u_{k},B) + \F_{\e_{k}}(v_{k},A \setminus \overline C)\,+\\
+ \frac{C_0}{M} \int_{B \setminus \overline C}( 1 + |\nabla u_{k}|^p
+ |\nabla v_{k}|^p + M^p
|u_{k} - v_{k}|^p)\, dx \Big\}
\end{multlined} \\
 &\leq
 \begin{multlined}[t][10cm]
 \F^{\{\e_k\}}_\K(u,B) + \F^{\{\e_k\}}_\K(u,A \setminus \overline C) +
2\eta\,+\\
+ \frac{C_0}{M} \,\sup_{k \in\Nb}\, \int_{B \setminus \overline
C}(1+|\nabla u_k|^p + |\nabla v_k|^p)\, dx\,.
\end{multlined}
\end{align*}
Then property (\ref{sub}) arises sending first $M \to +\infty$,
and then $\eta \to 0$.
\vskip5pt

{\bf Step 2.} Now we complete the proof of Lemma \ref{meas}. Using a
standard diagonal argument, we construct a subsequence
$\{\e_k\} \searrow 0^+$ and a sequence $\{u_k\}\subset
W^{1,p}(\O,\M)$ satisfying
\begin{multline*}
\lim_{k\to+\infty}\F_{\e_k}(u_k,\O)=\inf_{\{v_k\}} \Big\{ \liminf_{k \to +\infty}\,
\F_{\e_{k}} (v_k,\O) :  \, v_k \wto u \text{ weakly in }W^{1,p}(\O;\Rb^d)\,, \\
v_k \to u \text{ uniformly and } v_k(x)=u(x) \text{ whenever } {\rm
dist}\,(u(x),\K)>1\text{ for a.e. } x\in \O\, \Big\}\,.
\end{multline*}
By construction of $\{\e_k\}$ and $\{u_k\}$, we have $\ds  \lim_{k\to+\infty}\F_{\e_k}(u_k,\O)=\F^{\{\e_k\}}_\K(u,\O)$.
Up to the extraction of a further subsequence, we may assume that
$$ f\left(\frac{\cdot}{\e_{k}}, \nabla u_{k}\right)\LL^N\res\, \O\,\mathop{\rightharpoonup}\limits^{\star}\,\mu \quad\text{in }\M(\O)\,,$$
for some nonnegative Radon measure $\mu\in\M(\O)$. By lower
semicontinuity, we have
$$\mu(\O)\leq \lim_{k\to+\infty}\F_{\e_k}(u_k,\O)=\F^{\{\e_k\}}_\K(u,\O)\,. $$
We claim that
\begin{equation*}
\F^{\{\e_k\}}_\K(u,A)=\mu(A)\quad\text{ for any } A\in\A(\O)\,.
\end{equation*}
We fix $A\in\A(\O)$ and we start by proving the inequality ``$\leq$".
Given $\eta>0$ arbitrary we can select, in view of \eqref{absp}, $C\in\A(\O)$, $C\subset\subset A $, such that $\F^{\{\e_k\}}_\K(u,A\setminus\overline C)\leq \eta$.
Then inequality \eqref{sub} implies that for any $B\in\A(\O)$, $C\subset\subset B \subset \subset A$,
$$\F^{\{\e_k\}}_\K(u,A)\leq \eta+\limsup_{k\to+\infty}\,\F_{\e_k}(u_{k},B)\leq \eta+\mu(\overline B)\leq \eta+\mu(A)\,, $$
and the conclusion follows from the arbitrariness of $\eta$.

Conversely, for any $B\in\A(\O)$, $B\subset\subset A$, we have
\begin{multline*}
\mu(\O)\leq \F^{\{\e_k\}}_\K(u,\O)\leq \F^{\{\e_k\}}_\K(u,A)+\F^{\{\e_k\}}_\K(u,\O\setminus \overline B)
\leq \\
\leq \F^{\{\e_k\}}_\K(u,A)+\mu(\O\setminus \overline B)\leq \F^{\{\e_k\}}_\K(u,A)+\mu(\O) -\mu(B)\,.
\end{multline*}
Therefore $\mu(B)\leq \F^{\{\e_k\}}_\K(u,A)$ and the conclusion
follows by inner regularity of $\mu$.
\end{proof}

\section{The upper bound}

\noindent We now make use of the previous locality result to show
the upper bound. This will be done thanks to a blow-up analysis
in the spirit of \cite[Theorem~3.1]{DFMT}.

\begin{lemma}\label{upperbound}
For every $p \in [1,+\infty)$ and $u \in W^{1,p}(\O;\M)$, we have $\F(u)
\leq \F_{\rm hom}(u)\,$.
\end{lemma}

\begin{proof} {\bf Step 1.} Let $u \in W^{1,p}(\O;\M)$. Given $R>0$ arbitrary large, we set
$\K:= \M\cap B^d(0,R)$, and we consider the subsequence
$\{\e_k\}$ given by Lemma  \ref{meas}.  Obviously $\F(u)\leq \F^{\{\e_k\}}_\K(u,\O)$. We claim that
\begin{equation}\label{upbdR}
\F^{\{\e_k\}}_\K(u,\O)\leq \int_{\O}\bigg\{\chi_R(|u|)\,
Tf_{\rm hom}(u,\nabla u) +\beta\big(1-\chi_R(|u|)\big)\big(1+|\nabla u|^p\big)\bigg\}\,dx\,,
\end{equation}
where $\chi_R(t)=1$ for $t\leq R$ and $\chi_R(t)=0$ otherwise. We
postpone the proof of \eqref{upbdR} to the next step, and we complete
now the proof of Lemma \ref{upperbound}.

Consider a sequence $R_j\to+\infty$ as $j\to+\infty$. Since
$\chi_{R_j}\to 1$ pointwise, we deduce from Fatou's lemma together
with \eqref{pgT} that
\begin{multline*}
\F(u)  \leq
\limsup_{j\to+\infty} \int_{\O}\bigg\{\chi_{R_j}(|u|)\,Tf_{\rm
hom}(u,\nabla u)+ \beta\big(1-\chi_{R_j}(|u|)\big)\big(1+|\nabla
u|^p\big)\bigg\}\,dx \leq\\
\leq  \int_{\O}Tf_{\rm hom}(u,\nabla u)\, dx\,,
\end{multline*}
which is the announced estimate.
\vskip5pt

{\bf Step 2.}
Thanks to Lemma \ref{meas}, to obtain  \eqref{upbdR} it  suffices to prove that
\begin{multline*}
\frac{d\F^{\{\e_k\}}_\K(u,\cdot)}{d\LL^N}(x_0) \leq \chi_R(|u(x_0)|) \,Tf_\hom(u(x_0),\nabla
u(x_0))+\beta\big(1-\chi_R(|u(x_0)|)\big)\big(1+|\nabla u(x_0)|^p\big)
\end{multline*}
for $\LL^N$-a.e. $x_0 \in \O\,$.

Let $x_0 \in \O$ be a
Lebesgue point of $u$ and $\nabla u$ such that $u(x_0) \in \M$,
$\nabla u(x_0) \in [T_{u(x_0)}(\M)]^N$, and the Radon-Nikod\'ym
derivative of $\F^{\{\e_k\}}_\K(u,\cdot)$ with respect to the Lebesgue
measure $\LL^N$ exists. Note that almost every points in $\O$ satisfy
these properties. Now set $s_0:=u(x_0)$ and
$\xi_0:=\nabla u (x_0)$.
\vskip5pt

\noindent{\it Case 1.} Assume that $s_0\not\in\K$. Then, using
$(H_2)$,  we derive that
\begin{multline*}
\frac{d\F^{\{\e_k\}}_\K(u,\cdot)}{d\LL^N}(x_0)=\lim_{\rho\to0^+} \frac{\F^{\{\e_k\}}_\K(u,Q(x_0,\rho))}{\rho^N}\leq \limsup_{\rho\to0^+}\,\limsup_{k\to+\infty}\,\rho^{-N}\F_{\e_k}(u,Q(x_0,\rho))\leq \\
\leq \lim_{\rho\to0^+}\frac{\beta}{\rho^N}\int_{Q(x_0,\rho)}(1+|\nabla u|^p)\,dx=\beta\big(1+|\xi_0|^p\big)\,,
\end{multline*}
which is the desired estimate.
\vskip5pt

\noindent{\it Case 2.} Now we assume that $s_0\in\K$. Fix $0<\eta<1$ arbitrary. By Proposition \ref{properties1}, claim {\it (i)},  there exist $j \in \Nb$ and $\varphi \in
W^{1,\infty}_0((0,j)^N;T_{s_0}(\M))$ such that
\begin{equation}\label{fhomT}
\med_{(0,j)^N} f(y,\xi_0 + \nabla \varphi(y))\, dy \leq
Tf_\hom(s_0,\xi_0) + \eta\,.
\end{equation}
Extend $\varphi$ to $\Rb^N$
by $j$-periodicity, and define $\varphi_k(x):= \xi_0 \, x + \e_k
\varphi(x/\e_k)$.
\vskip5pt
Let $\mathcal{U}$ be an open neighborhood of $\M$ such that the nearest point projection $\Pi:\mathcal{U}\to\M$
defines a $\mathcal{C}^1$-mapping. Fix $\sigma,\delta_0\in (0,1)$ such that $B^d(s_0,2\delta_0)\subset \mathcal{U}$, and consider
 $\delta=\delta(\sigma)\in(0,\delta_0)$ for which
\begin{equation}\label{unifcontPi}
|\nabla \Pi(s) - \nabla \Pi(s')|<\sigma \quad\text{for all $s,s' \in B^d(s_0,\d_0)$ satisfying  $|s-s'|<\delta\,$.}
\end{equation}
Next we introduce a cut-off function $\zeta \in \C^\infty_c(\Rb^d;[0,1])$
satisfying
$$\zeta(x)=\left\{
\begin{array}{ll}
1 & \text{ for } x\in B^d(0,\d/4)\,,\\
&\\
0 & \text{ for } x\not\in B^d(0,\d/2)\,,\\
\end{array}
\right. \quad\text{with}\quad |\nabla \zeta|\leq \frac{C}{\d}\,,$$
and we define
$$w_k(x):=u(x)+\e_k\zeta(u(x)-s_0)\varphi(x/\e_k)\,. $$
Let
$k_0 \in \Nb$ be such that
\begin{equation}\label{k_0}\max\left\{\e_k \|\varphi\|_{L^\infty((0,j)^N;\Rb^d)}
\|\nabla \zeta\|_{L^\infty(\Rb^d;\Rb^d)}, \frac{2\e_k
\|\varphi\|_{L^\infty((0,j)^N;\Rb^d)}}{\delta}\right\} < 1 \quad
\text{ for any }k \geq k_0\,.
\end{equation}
Define for every $k \geq k_0$,
$$u_k(x):=\Pi(w_k(x)) \,.$$
Remark that by (\ref{k_0}), for a.e. $x\in\O\,$
and all $k\geq k_0$, one has $w_k(x) \in B^d(s_0,\d)$ whenever
$|u(x)-s_0| <\d/2$ while $w_k(x)=u(x)$ when $|u(x)-s_0| \geq\d/2$.
Hence $u_k$ is well defined,  $\{u_k\} \subset
W^{1,p}(\O;\M)$, and for a.e. $x\in\O$, $u_k(x)=u(x)$ whenever ${\rm
dist}\,(u(x),\K)>1$. Moreover,
\begin{multline*}
\|u_k - u\|_{L^\infty(\O;\Rb^d)}  =  \|\Pi(w_k) - \Pi(u)\|_{L^\infty(\{|u-s_0|<\d/2\};\Rb^d)}
 \leq\\
 \leq \e_k \, \|\nabla
\Pi\|_{L^\infty(B^d(s_0,\d_0);\Rb^d)}\|\varphi\|_{L^\infty((0,j)^N;\Rb^d)}
\to 0
\end{multline*}
as $k \to +\infty$. Now
the Chain Rule formula yields
\begin{multline*}
\nabla u_k(x)=\nabla \Pi(w_k(x)) \Big( \nabla u(x) + \e_k
\big(\varphi(x/\e_k) \otimes \nabla \zeta (u(x)-s_0)\big) \nabla
u(x)\, +\\
+ \zeta (u(x)-s_0) \nabla \varphi(x/\e_k) \Big)\,,
\end{multline*}
and consequently
\begin{multline*}
|\nabla u_k(x)|  \leq  \|\nabla \Pi\|_{L^\infty(B^d(s_0,\d_0);\Rb^d)}
\Big( \big(1+ \e_k \|\varphi\|_{L^\infty((0,j)^N;\Rb^d)} \|\nabla
\zeta\|_{L^\infty(\Rb^d;\Rb^d)} \big) |\nabla u(x)|
\,+\\
+ \|\nabla \varphi\|_{L^\infty((0,j)^N;\Rb^{d \times
N})} \Big)\,.
\end{multline*}
By (\ref{k_0}) it follows that for any $k \geq k_0$,
\begin{eqnarray}\label{gradu_n}|\nabla u_k(x)| & \leq & C_0
(|\nabla u(x) - \xi_0| +1)
\end{eqnarray}
for some constant $C_0=C_0(s_0,\xi_0,\delta_0,\eta)>0$ independent
of $x$ and $k$. Hence the sequence $\{u_k\}$ is uniformly bounded in
$W^{1,p}(\O;\Rb^d)$ so that $u_k\rightharpoonup u$ in
$W^{1,p}(\O;\Rb^d)$.

Then we observe that $|\nabla u_k|\leq 2C_0$ a.e. in $\{|\nabla u -
\xi_0| < \sigma\}$ while $$\|\nabla \varphi_k\|_{L^\infty(\O;\Rb^{d
\times N})}\leq |\xi_0|+ \|\nabla \varphi\|_{L^\infty((0,j)^N;\Rb^{d
\times N})}\,.$$
Set
\begin{equation}\label{defM}
M:=\max\big\{2C_0\,,|\xi_0|+\|\nabla
\varphi\|_{L^\infty((0,j)^N;\Rb^{d \times N})}\big\}\,,
\end{equation}
(which only depends on $s_0$, $\xi_0$, $\delta_0$ and $\eta$)
so that
\begin{equation}\label{bound}
|\nabla u_k| \leq M \quad \text{and}\quad |\nabla \varphi_k|
\leq M \quad \text{ a.e. in $\{|\nabla u -
\xi_0| < \sigma\}$\,.}
\end{equation}
Next for a.e. $x \in \{|u-s_0| < \d/4\}\cap\{|\nabla u -
\xi_0|< \sigma\}$, we have $\zeta(u(x)-s_0)=1$ and
\begin{align*}
|\nabla u_k(x) - \nabla \varphi_k(x)| & \leq |\nabla
\Pi(w_k)\nabla u(x) - \xi_0| + |\nabla \Pi(w_k) \nabla
\varphi(x/\e_k) - \nabla
\varphi(x/\e_k)|\\
& \leq  \begin{multlined}[t][10cm]
\!  |\nabla \Pi(w_k) - \nabla \Pi(s_0)| \, |\nabla u(x)| +
|\nabla \Pi(s_0)|\, |\nabla u(x)-\xi_0|+\\
+|\nabla \Pi(w_k) - \nabla \Pi(s_0)|\, \|\nabla
\varphi\|_{L^\infty((0,j)^N;\Rb^{d \times N})}\,,
\end{multlined}
\end{align*}
where, in the last inequality, we have used the fact that
 $\nabla \Pi(s_0)\nabla
\varphi(y)=\nabla \varphi(y)$ since $\nabla
\varphi(y) \in [T_{s_0}(\M)]^N$  for a.e. $y \in \Rb^N$.  Using  (\ref{unifcontPi}) and the fact that
$|w_k-s_0|<\d$ a.e. in $\{|u-s_0| < \d/4\}\cap\{|\nabla u -
\xi_0|< \sigma\}$, we deduce
\begin{eqnarray}\label{diff}
|\nabla u_k(x) - \nabla \varphi_k(x)| \leq \big( |\nabla u(x)|  +
|\nabla \Pi(s_0)|+\|\nabla \varphi\|_{L^\infty((0,j)^N;\Rb^{d \times
N})} \big) \sigma \leq C_1 \sigma
\end{eqnarray}
for a.e. $x \in \{|u-s_0| < \d/4\}\cap\{|\nabla u - \xi_0|<
\sigma\}$, where $C_1=C_1(s_0,\xi_0,\delta_0,\eta)>0$ is a constant
independent of $\sigma$, $k$ and $x$.

Now we estimate
\begin{align}\label{I}
\nonumber \frac{d\F^{\{\e_k\}}_\K(u,\cdot)}{d\LL^N}\,&(x_0)  =  \lim_{\rho \to 0^+}
\frac{\F^{\{\e_k\}}_\K(u,Q(x_0,\rho))}{\rho^N}\nonumber\\
\nonumber \leq &\,\limsup_{\rho \to 0^+}\limsup_{k\to +\infty}\,
\frac{1}{\rho^N}
\int_{Q(x_0,\rho)}f\left(\frac{x}{\e_k},\nabla u_k\right) dx\\
\nonumber \leq &
\, \limsup_{\rho \to 0^+}\limsup_{k \to +\infty}
\frac{1}{\rho^N} \int_{Q(x_0,\rho) \cap
\{|u-s_0|\geq \d/4\}}f\left(\frac{x}{\e_k},\nabla u_k\right) dx\,\\
\nonumber&\,+ \limsup_{\rho \to 0^+}\limsup_{k \to +\infty}\,
\frac{1}{\rho^N}\int_{Q(x_0,\rho) \cap \{|u-s_0|<\d/4\} \cap \{|\nabla u -
\xi_0| < \sigma\}}  f\left(\frac{x}{\e_k},\nabla u_k\right) dx\,\\
\nonumber&\,+ \limsup_{\rho \to 0^+}\limsup_{k \to +\infty}\,
\frac{1}{\rho^N} \int_{Q(x_0,\rho) \cap \{|u-s_0|<\d/4\} \cap \{|\nabla u -
\xi_0| \geq \sigma\}} f\left(\frac{x}{\e_k},\nabla u_k\right) dx\\
 =:&\; I_1 + I_2 + I_3\,.
\end{align}
Thanks to (\ref{gradu_n}), the $p$-growth condition $(H_2)$ and our
choice of $x_0$, we have
\begin{align}\label{I1}
I_1 & \leq  C \limsup_{\rho \to 0^+} \frac{1}{\rho^N}
\int_{Q(x_0,\rho) \cap
\{|u-s_0|\geq \d/4\}}(1+|\nabla u(x) -\xi_0|^p)\, dx\nonumber\\
 &
 \begin{multlined}[10cm]
 \leq  C \limsup_{\rho \to 0^+} \med_{Q(x_0,\rho)} |\nabla u(x) -\xi_0|^p\, dx
+  \frac{4C}{\d}\limsup_{\rho \to 0^+} \med_{Q(x_0,\rho)}
|u(x)-s_0|\, dx=0\,,
\end{multlined}
\end{align}
while
\begin{align}\label{I3}
\nonumber I_3 & \leq C \limsup_{\rho \to 0^+} \frac{1}{\rho^N}
\int_{Q(x_0,\rho) \cap \{|u-s_0|<\d/4\} \cap \{|\nabla u -
\xi_0| \geq \sigma\}}(1+|\nabla u(x) -\xi_0|^p)\, dx\\
&
\begin{multlined}[10cm]
 \leq  C \limsup_{\rho \to 0^+} \med_{Q(x_0,\rho)} |\nabla u(x)
-\xi_0|^p\, dx
+ \frac{C}{\sigma}\limsup_{\rho \to 0^+}
\med_{Q(x_0,\rho)} |\nabla u(x)-\xi_0|\, dx=0\,.
\end{multlined}
\end{align}

Let us now treat the integral $I_2$. Since, for a.e. $y \in \Rb^N$,
the function $f(y,\cdot)$ is continuous, it is uniformly continuous
on $B^{d \times N}(0,M)$ where $M>0$ is given in (\ref{defM}).
Define the modulus of continuity of $f(y,\cdot)$ over $B^{d \times
N}(0,M)$ by
$$\o(y,t):= \sup \{ |f(y,\xi) - f(y,\xi')| : \; \xi,\, \xi' \in
B^{d \times N}(0,M) \text{ and }|\xi-\xi'| \leq t\}\,.$$
It turns out
that $\o(y,\cdot)$ is increasing, continuous and $\o(y,0)=0$, while
$\o(\cdot,t)$ is measurable (since the supremum can be restricted to all admissible $\xi$ and $\xi'$ having rational
entries) and $1$-periodic. Thanks to (\ref{bound}) and
(\ref{diff}) we get that
$$\left|f\left(\frac{x}{\e_k},\nabla u_k(x) \right) - f\left(\frac{x}{\e_k},\nabla \varphi_k(x)
\right)\right| \leq \o\left(\frac{x}{\e_k},C_1\sigma\right)$$ for
a.e. $x \in Q(x_0,\rho)\cap\{|u-s_0| < \d/4\}\cap\{|\nabla u- \xi_0|
< \sigma\}$.

Integrating over the set $Q(x_0,\rho) \cap \{|u-s_0|<\d/4\} \cap
\{|\nabla u - \xi_0| < \sigma\}$, and taking the limit as $k\to
+\infty$, we obtain in view of the Riemann-Lebesgue Lemma that
\begin{multline*}
\!\!\limsup_{k \to +\infty}\rho^{-N}\int_{Q(x_0,\rho) \cap
\{|u-s_0|<\d/4\} \cap \{|\nabla u - \xi_0| <
\sigma\}}\left|f\left(\frac{x}{\e_k},\nabla u_k(x) \right) -
f\left(\frac{x}{\e_k},\nabla \varphi_k(x) \right)\right|dx\leq \\
\leq  \limsup_{k \to +\infty}\rho^{-N} \int_{Q(x_0,\rho)}
\o\left(\frac{x}{\e_k},C_1\sigma\right)dx =  \int_Q
\o(y,C_1\sigma)\, dy\,,
\end{multline*}
where we have used the fact that $y \mapsto \o(y,C_1\sigma)$ is a
measurable $1$-periodic function. Observe that the Dominated
Convergence Theorem together with $ \o(y,0)=0$ implies
\begin{equation}\label{vanom}
\lim_{\sigma \to 0^+}\int_Q \o(y,C_1\sigma)\, dy=0\,.
\end{equation}
We have obtained
\begin{equation}\label{I2}
I_2  \leq  \limsup_{\rho \to 0^+} \limsup_{k \to +\infty}
\frac{1}{\rho^N} \int_{Q(x_0,\rho)} f\left(\frac{x}{\e_k},\nabla
\varphi_k \right) dx+\int_Q \o(y,C_1\sigma)\, dy\,.
\end{equation}
Using the definition of $\varphi_k$ and the Riemann-Lebesgue Lemma,
we infer from (\ref{fhomT}) that
\begin{multline}\label{I2bis}
 \limsup_{\rho \to 0^+} \limsup_{k \to +\infty}
\frac{1}{\rho^N} \int_{Q(x_0,\rho)} f\left(\frac{x}{\e_k}, \xi_0 +
\nabla
\varphi\left(\frac{x}{\e_k}\right) \right) dx=\\
 =  \med_{(0,j)^N} f(y,\xi_0 + \nabla \varphi(y))\, dy
 \leq  Tf_\hom(s_0,\xi_0) + \eta\,.
\end{multline}
Hence gathering (\ref{I}), (\ref{I1}), (\ref{I3}), (\ref{I2}) and \eqref{I2bis} we
deduce that
$$\frac{d\F^{\{\e_k\}}_\K(u,\cdot)}{d\LL^N}(x_0) \leq  Tf_\hom(s_0,\xi_0) +\int_Q \o(y,C_1\sigma)\, dy+
\eta\,.$$
Thanks to \eqref{vanom}, the thesis follows sending first $\sigma\to0$, and then $\eta\to0$.
\end{proof}

\section{The lower bound}

\noindent We now investigate the $\G$-$\liminf$ inequality still
through the blow-up method. In contrast with Lemma \ref{upperbound}
we will distinguish energies with superlinear growth and energies
with linear growth. We will conclude this section with the proofs of
Theorems \ref{babmil} and \ref{babmilp=1}.

\subsection{The case of superlinear growth}

\noindent The case $p>1$ is based on an equi-integrability
result known as Decomposition Lemma \cite[Lemma~1.2]{FMP}, which
allows to consider sequences with $p$-equi-integrable
gradients. It enables to use properties valid up to
sets where the energy remains  small.

\begin{lemma}\label{lowerbound}
Assume $p \in (1,+\infty)$.  Then  $\F(u) \geq \F_{\rm hom}(u)$ for every $u \in W^{1,p}(\O;\M)\,$.
\end{lemma}

\begin{proof}
Let $u \in W^{1,p}(\O;\M)$. By a standard diagonal argument, we first obtain a subsequence $\{\e_n\}$ (not relabeled) and
$\{u_n\} \subset W^{1,p}(\O;\M)$ such that $u_n\to  u$ in $L^p(\O;\Rb^d)$ and
$$\F(u)= \lim_{n \to
+\infty}\int_\O f \left(\frac{x}{\e_n},\nabla u_n \right)dx<+\infty\,.$$
Define the sequence of nonnegative Radon measures
$$\mu_n:=f \left(\frac{\cdot}{\e_n},\nabla u_n \right)\LL^N
\res\,  \O\,.$$ Extracting a  further subsequence if necessary,
we may assume that there exists a nonnegative Radon measure $\mu \in
\M(\O)$ such that $\mu_n \xrightharpoonup[]{*} \mu$ in $\M(\O)$.
Using Lebesgue Differentiation Theorem one can split $\mu$ into
the sum of two mutually disjoint nonnegative measures $\mu=\mu^a +
\mu^s$ where $\mu^a \ll \mathcal L^N$ and $\mu^s$ is singular with
respect to $\LL^N$. Since $\mu^a(\O)\leq\mu(\O) \leq \F(u)$, it is
enough to check that
$$\frac{d\mu}{d\LL^N}(x_0) \geq Tf_\hom(u(x_0),\nabla u(x_0))\quad
\text{ for }\LL^N\text{-a.e. }x_0 \in \O\,.$$
\vskip5pt

{\bf Step 1.} Select a point $x_0 \in
\O$ which is a Lebesgue point of $u$ and $\nabla u$,  a
point of approximate differentiability of $u$ (so that $u(x_0) \in
\M$, $\nabla u(x_0) \in [T_{u(x_0)}(\M)]^N$), and such that the
Radon-Nikod\'ym derivative of $\mu$ with respect to the Lebesgue
measure $\LL^N$ exists and is finite. Note that almost every points
$x_0$ in $\O$ satisfy these properties. As in the proof of Lemma
\ref{upperbound}, set $s_0:=u(x_0)$ and $\xi_0:=\nabla u(x_0)$.

Let $\{\rho_k\} \searrow 0^+$ be such that $\mu(\partial
Q(x_0,\rho_k))=0$ for every $k \in \Nb$. Using the integrand
$\bar f$ defined in \eqref{deffbar} one obtains
\begin{align*}
\nonumber\frac{d\mu}{d\LL^N}(x_0) &= \lim_{k \to +\infty}
\frac{\mu(Q(x_0,\rho_k))}{\rho_k^N}\\
\nonumber&=\lim_{k \to +\infty}
\lim_{n \to +\infty}\frac{\mu_n(Q(x_0,\rho_k))}{\rho_k^N}\\
\nonumber&= \lim_{k \to +\infty} \lim_{n \to +\infty} \int_Q
f\left(\frac{x_0 +\rho_k y}{\e_n},\nabla
u_n(x_0 + \rho_k y)\right)dy\\
\nonumber&= \lim_{k \to +\infty} \lim_{n \to +\infty} \int_Q
\bar f\left(\frac{x_0 +\rho_k y}{\e_n},u_n(x_0+\rho_k y),\nabla
u_n(x_0 + \rho_k y)\right)dy\\
&=\lim_{k \to +\infty} \lim_{n \to +\infty} \int_Q
\bar f\left(\frac{x_0 +\rho_k y}{\e_n},s_0+\rho_kv_{n,k}(y),\nabla v_{n,k}(y)\right)dy\,,
\end{align*}
where we have set $v_{n,k}(y):=\big[u_n(x_0 + \rho_k y) - s_0
\big]/\rho_k$. Note that since $x_0$ is a point of approximate
differentiability of $u$ and $u_n \to u$ in $L^p(\O;\Rb^d)$, we have
\begin{multline*}
\lim_{k \to +\infty} \lim_{n \to +\infty} \int_Q
|v_{n,k}(y) - \xi_0\, y|^p\, dy =\lim_{k \to +\infty}
\int_{Q(x_0,\rho_k)} \frac{|u(y) - s_0 - \xi_0 \,
(y-x_0)|^p}{\rho_k^{N+p}}\, dy= 0\,.
\end{multline*}
Hence
one can find a
diagonal sequence $\e_k:=\e_{n_k} < \rho_k^2$ such that, setting
$v_k(y):=v_{n_k,k}(y)$ and $v_0(y):=\xi_0\, y$,  $v_k \to v_0$
in $L^p(Q;\Rb^d)$ and
\begin{equation}\label{inerst}
\frac{d\mu}{d\LL^N}(x_0) = \lim_{k \to +\infty} \int_Q
\bar f\left(\frac{x_0 +\rho_k y}{\e_{k}}, s_0+\rho_k
v_k(y),\nabla v_k(y)\right)dy\,.\end{equation}
Next observe that $\{\nabla v_k\}$ is bounded in $L^p(Q;\Rb^{d \times
N})$ thanks to the coercivity condition \eqref{growthfbar}.
By the Decomposition Lemma \cite[Lemma~1.2]{FMP} we
now find a sequence $\{\bar v_k\} \subset
W^{1,\infty}(Q;\Rb^d)$ such that $\bar v_k=v_0$ on a neighborhood of $\partial Q$,
$\bar v_k \to v_0$
in $L^p(Q;\Rb^d)$,
the sequence of gradients $\{|\nabla \bar v_k|^p\}$ is
equi-integrable, and
\begin{multline}\label{inerst2}
\lim_{k \to +\infty} \int_Q \bar f\left(\frac{x_0 +\rho_k
y}{\e_{k}}, s_0+\rho_k v_k(y),\nabla v_k(y)\right)dy\\ \geq
\limsup_{k \to +\infty}\, \int_Q \bar f\left(\frac{x_0 +\rho_k
y}{\e_{k}}, s_0+\rho_k v_k(y),\nabla \bar v_k(y)\right)dy\,.
\end{multline}
\vskip5pt

{\bf Step 2.} Write
$$\frac{x_0}{\e_{k}}=m_k+s_k \quad \text{with } m_k \in \Zb^N
\text{ and } s_k \in [0,1)^N\,,$$ and define
\begin{equation*}
x_k:=\frac{\e_{k}}{\rho_k}s_k \to
0\quad\text{and}\quad\d_k:=\e_{k}/\rho_k\to 0\,.
\end{equation*}
By the $1$-periodicity of $\bar f$ with respect
to its first variable, \eqref{inerst} and \eqref{inerst2}, we infer
\begin{eqnarray}\label{debut}
\frac{d\mu}{d\LL^N}(x_0) & \geq & \limsup_{k \to +\infty}\, \int_Q \bar
f\left(\frac{x_k +y}{\d_k}, s_0+\rho_k v_k(y),\nabla
\bar v_k(y)\right)dy\nonumber\\
& \geq &  \limsup_{k \to +\infty}\, \int_{x_k+Q} \bar
f\left(\frac{y}{\d_k}, s_0+\rho_k v_k(y - x_k),\nabla
\bar v_k(y-x_k)\right)dy\,.
\end{eqnarray}
Extend $v_k$ by $0$, and $\bar v_k$ by $v_0$ to the whole $\Rb^N$. As $x_k
\to 0$ it follows that $\LL^N((Q-x_k) \triangle Q) \to 0$, and the
equi-integrability of $\{|\nabla \bar v_k|^p\}$ together with the
$p$-growth condition (\ref{growthfbar}) implies
$$\int_{Q \triangle (x_k+Q)}\bar
f\left(\frac{y}{\d_k}, s_0+\rho_k v_k(y - x_k),\nabla
\bar v_k(y-x_k)\right)dy \leq \b'\int_{(Q-x_k) \triangle Q}(1+|\nabla
\bar v_k|^p)\, dy \to 0\,.$$
Hence (\ref{debut}) yields
\begin{equation}\label{debut1}
\frac{d\mu}{d\LL^N}(x_0) \geq \limsup_{k \to +\infty}\, \int_Q \bar
f\left(\frac{y}{\d_k}, s_0+\rho_k w_k,\nabla \bar w_k \right)dy\,,
\end{equation}
where $w_k(y):=v_k(y-x_k)$ and $\bar w_k(y):=\bar v_k(y-x_k)$ converge to $v_0$ in $L^p(Q;\Rb^d)$,
and $\{|\nabla \bar w_k|^p\}$ is equi-integrable as well.
\vskip5pt

{\bf Step 3.} For $M>1$ and $k \in\Nb$, consider the set $E_{M,k}:=\{x \in Q :
|\nabla \bar w_k|\leq M\}$. By Chebyschev inequality, (\ref{debut1})
and (\ref{growthfbar}),  $\LL^N(Q
\setminus E_{M,k}) \leq C/M^p$ for some constant $C>0$ independent
of $k$ and $M$.

By the Scorza-Dragoni Theorem (see \cite{ET}, p. 235), for any
$\eta>0$, we may find a compact set $K_\eta \subset \overline Q$
such that $\LL^N(\overline Q\setminus K_\eta)<\eta$ and $f:K_\eta
\times \Rb^{d \times N} \to [0,+\infty)$ is continuous. In
particular the restriction of $\bar f(\cdot,s,\cdot)$ to $K_\eta
\times B^{d \times N}(0,M)$ is uniformly continuous for every
$s\in\M$.  Therefore the function $\Psi_{\eta,M} : [0,+\infty) \to
[0,+\infty)$ defined by
$$\Psi_{\eta,M}(t)=\sup\bigg\{|f(y,\xi)-f(y,\xi')|\; :\; y\in K_\eta,\;
\xi,\,\xi'\in B^{d\times N}(0,M)\,,\; |\xi-\xi'|\leq t\bigg\}\,,$$
is continuous, satisfies $\Psi_{\eta,M}(0)=0$, and is bounded. In view of
\eqref{deffbar}, we have
$$|\bar f(y,s_1,\xi)-\bar
f(y,s_2,\xi)|\leq
\Psi_{\eta,M}\big(M|\mathbf{P}_{s_1}-\mathbf{P}_{s_2}|\big)+C_M|\mathbf{P}_{s_1}-\mathbf{P}_{s_2}|
=:\tilde \Psi_{\eta,M}
\big(|\mathbf{P}_{s_1}-\mathbf{P}_{s_2}|\big)$$
for every $y\in
K_\eta$, $s_1,s_2\in \M$ and $\xi\in B^{d \times N}(0,M)$, where the
constant $C_M>0$ only depends on $M$ and $p$. Define
$$K_\eta^{\rm per}:= \bigcup_{\ell\in\Zb^N}\big(\ell+K_\eta\big)\,.$$
Since $\bar f$ is 1-periodic in the first variable,
\begin{equation}\label{conts}
|\bar f(y,s_1,\xi)-\bar f(y,s_2,\xi)|\leq \tilde \Psi_{\eta,M}
\big(|\mathbf{P}_{s_1}-\mathbf{P}_{s_2}|\big)
\end{equation}
for every $y\in K_\eta^{\rm per}$, $s_1,s_2\in\M$ and $\xi\in
B^{d\times N}(0,M)\,$. From (\ref{debut1}) and (\ref{conts}) it
follows that
\begin{align*}
\frac{d\mu}{d\LL^N}(x_0) & \geq  \limsup_{k \to +\infty}
\int_{E_{M,k} \cap (\d_k K^{\rm per}_\eta) } \bar
f\left(\frac{y}{\d_k},
s_0+\rho_k w_k,\nabla \bar w_k \right)dy\\
&
\begin{multlined}[10cm]
\geq \limsup_{k \to +\infty} \int_{E_{M,k} \cap (\d_k K^{\rm
per}_\eta) } \bar f\left(\frac{y}{\d_k}, s_0,\nabla \bar w_k
\right)dy \,-\\
-\limsup_{k \to +\infty} \int_Q \tilde \Psi_{\eta,M}
\big(|\mathbf{P}_{s_0 + \rho_k w_k(y)}-\mathbf{P}_{s_0}|\big)\, dy\,.
\end{multlined}
\end{align*}
Since $\tilde \Psi_{\eta,M}$ is continuous and bounded, $\tilde
\Psi_{\eta,M}(0)=0$, and (up to a subsequence) $\mathbf{P}_{s_0+\rho_k w_k(y)} \to
\mathbf{P}_{s_0}$ for a.e. $y\in Q$, we obtain by Dominated
Convergence that
$$\lim_{k\to+\infty} \int_{Q}\tilde \Psi_{\eta,M} \big(|\mathbf P_{s_0+\rho_k
w_k(y)}-\mathbf P_{s_0}|\big)\, dy=0\,,$$ and thus
\begin{equation}\label{debut2}
\frac{d\mu}{d\LL^N}(x_0) \geq \limsup_{k \to +\infty}
\int_{E_{M,k} \cap (\d_k K^{\rm per}_\eta) } \bar
f\left(\frac{y}{\d_k}, s_0,\nabla \bar w_k \right)dy\,.
\end{equation}
From
the $p$-growth condition (\ref{growthfbar}) and the Riemann-Lebesgue
Lemma, we deduce that
\begin{multline*}
 \limsup_{k \to +\infty}
\int_{E_{M,k} \setminus (\d_k K_\eta^{\rm per})}\bar f\left(\frac{y}{\d_k}, s_0,\nabla \bar w_k\right)dy   \leq  \limsup_{k
\to +\infty} \,\b'(1+M^p)\LL^N(Q \setminus (\d_k K^{\rm
per}_\eta))=\\
 =  \b'(1+M^p)\LL^N(Q \setminus K_\eta) \leq
\b'(1+M^p)\eta\,.\end{multline*}
Hence (\ref{debut2}) yields
$$\frac{d\mu}{d\LL^N}(x_0) \geq \limsup_{k \to +\infty}
\int_{E_{M,k}} \bar f\left(\frac{y}{\d_k}, s_0,\nabla \bar w_k
\right)dy - \b'(1+M^p)\eta\,,$$ and sending $\eta \to 0$, we derive
\begin{equation}\label{debut3}\frac{d\mu}{d\LL^N}(x_0) \geq \limsup_{k \to +\infty}
\int_{E_{M,k}} \bar f\left(\frac{y}{\d_k}, s_0,\nabla \bar w_k
\right)dy\,.
\end{equation}
Since $\LL^N(Q \setminus E_{M,k})\to 0$ as
$M\to +\infty$ (uniformly with respect to $k$), the
equi-integrability of $\{|\nabla \bar w_k|^p\}$ and the $p$-growth
condition (\ref{growthfbar}) imply
$$\sup_{k\in\Nb}\,\int_{Q \setminus E_{M,k}}\bar
f\left(\frac{y}{\d_k}, s_0,\nabla \bar w_k\right)dy \leq \b'\sup_{k\in\Nb}\,\int_{Q
\setminus E_{M,k}}(1+|\nabla \bar w_k|^p)\, dy \longrightarrow  0\quad\text{as $M\to+\infty$}\,.$$
Plugging this estimate in
(\ref{debut3}) leads to $$\frac{d\mu}{d\LL^N}(x_0) \geq
\limsup_{k \to +\infty} \int_Q \bar f\left(\frac{y}{\d_k},
s_0,\nabla \bar w_k \right)dy\,.$$
Since $\bar w_k\to v_0$ in $L^p(Q;\Rb^d)$, we can
invoke standard homogenization results (see, {\it e.g., }
\cite[Theorem~14.5]{BD}) to infer that
$$\limsup_{k \to +\infty} \int_Q \bar f\left(\frac{y}{\d_k},
s_0,\nabla \bar w_k \right)dy\geq \int_Q \bar f_{\rm hom}(s_0,\nabla
v_0)\, dy=\bar f_{\rm hom}(s_0,\xi_0)\,.$$ In view of Proposition
\ref{properties1} we finally conclude
$$\frac{d\mu}{d\LL^N}(x_0)\geq \bar f_{\rm hom}(s_0,\xi_0)= T f_{\rm hom}(s_0,\xi_0)\,,$$
and the proof is complete.
\end{proof}

\subsection{The case of linear growth}

\noindent We now treat the case $p=1$ assuming that the function $u$
belongs to  $W^{1,1}(\O;\M)$.
In contrast with
the case $p>1$, there is no equi-integrability result as the
Decomposition Lemma. We follow here the  approach of~\cite{FM}.

\begin{lemma}\label{lowerbound1}
Assume $p=1$. Then $\F(u) \geq \F_{\rm hom}(u)$ for every $u \in W^{1,1}(\O;\M)$.
\end{lemma}

\begin{proof} Let $u \in W^{1,1}(\O;\M)$. By a standard diagonal argument, we first obtain a subsequence $\{\e_n\}$ (not relabeled) and
$\{u_n\} \subset W^{1,1}(\O;\M)$ such that $u_n\to  u$ in $L^1(\O;\Rb^d)$ and
$$\F(u)=\lim_{n \to +\infty}\int_\O f\left(\frac{x}{\e_n},\nabla
u_n\right) dx<+\infty\,.$$
Up to the extraction of a further subsequence, we may assume that there
exists a nonnegative Radon measure $\mu \in \M(\O)$ such that
\begin{equation}\label{w*conv}
f\left(\frac{\cdot}{\e_n},\nabla u_n\right)\LL^N\res\, \O
\xrightharpoonup[]{*} \mu \quad\text{in }\M(\O)\,.
\end{equation}
Hence it
is enough to prove that $\mu(\O) \geq \F_\hom(u)$. As in the proof of Lemma \ref{lowerbound}, it suffices to show  that
\begin{equation*}
\frac{d\mu}{d\LL^N}(x_0) \geq Tf_\hom(u(x_0),\nabla u(x_0))\quad\text{for $\LL^N$-a.e. $x_0 \in \O$}\,.
\end{equation*}
The proof will be divided into three
steps. We first apply the blow-up method which reduces
the study to affine limiting functions. Then we reproduce the argument
of \cite{FM} which enables us to replace the original
sequence by a uniformly converging one without increasing the
energy. We will conclude using a
classical homogenization result.
\vskip5pt

{\bf Step 1.} Select a point $x_0 \in \O$ which is a Lebesgue point
of $u$ and $\nabla u$, a point of approximate differentiability of
$u$ (so that $u(x_0) \in \M$, $\nabla u(x_0) \in
[T_{u(x_0)}(\M)]^N$) and such that the Radon-Nikod\'ym derivative of
$\mu$ with respect to the Lebesgue measure $\LL^N$ exists and is
finite. Note that $\LL^N$-almost every points $x_0$ in $\O$ satisfy
these properties. We write  $s_0:=u(x_0)$ and $\xi_0:=\nabla
u(x_0)$.

Up to a subsequence, we may assume that there exists a nonnegative
Radon measure $\lambda \in \M(\O)$ such that $(1+|\nabla u_n|) \LL^N
\res\, \O \xrightharpoonup[]{*} \lambda$ in $\M(\O)$. Consider a
sequence $\{\rho_k\} \searrow 0^+$ such that
$Q(x_0,2\rho_k)\subset\O$ and $\mu(\partial Q(x_0,\rho_k))
=\lambda(\partial Q(x_0,\rho_k))=0$ for each $k \in \Nb$. Then
(\ref{w*conv}) yields
\begin{eqnarray}\label{1116}
\mu(Q(x_0,\rho_k)) = \lim_{n \to +\infty} \int_{Q(x_0,\rho_k)} f
\left(\frac{x}{\e_n},\nabla u_n \right)dx\,.
\end{eqnarray}
Set $\ds \tau_n:= \e_n \big[x_0/\e_n\big]\in\e_n\Zb^N$. Since $\tau_n \to x_0$, given $r\in(1,2)$
we have $Q(\tau_n,\rho_k) \subset Q(x_0,r\rho_k)$ whenever $n$ is large enough,
and we may define for $x\in Q(0,\rho_k)$, $v_n(x):=u_n(x+\tau_n)$.
By continuity of the translation in
$L^1$, we get that
\begin{align}\label{convv}
\int_{Q(0,\rho_k)}|v_n(x)-u(x+x_0)|\, dx & =
\int_{Q(\tau_n,\rho_k)}|u_n(x)-u(x+x_0-\tau_n)|\, dx\nonumber\\
& \leq  \int_{Q(x_0,r\rho_k)}|u_n(x)-u(x+x_0-\tau_n)|\, dx
\mathop{\longrightarrow}\limits_{n \to +\infty}0\,.
\end{align}
Changing variable in (\ref{1116}) and using the periodicity
condition $(H_1)$ of $f(\cdot,\xi)$ and the growth condition
$(H_2)$, we are led to
\begin{align}\label{lastterm}
\mu(Q(x_0,\rho_k)) &=  \lim_{n \to
+\infty} \int_{Q(x_0-\tau_n,\rho_k)} f\left(\frac{x+\tau_n}{\e_n},
\nabla
u_n(x+\tau_n)\right)dx\nonumber\\
&=  \lim_{n \to +\infty} \int_{Q(x_0-\tau_n,\rho_k)}
f\left(\frac{x}{\e_n}, \nabla v_n\right)dx\nonumber\\
&
\begin{multlined}[10cm]
\geq  \limsup_{n \to +\infty} \int_{Q(0,\rho_k)}
f\left(\frac{x}{\e_n},\nabla v_n\right)dx\,-\\
-\b \limsup_{n \to +\infty}
\int_{Q(\tau_n,\rho_k) \setminus Q(x_0,\rho_k)}(1+|\nabla u_n|)\,
dx\,.
\end{multlined}
\end{align}
On the other hand, by our choice of $\rho_k$,
\begin{align*}
\limsup_{n \to +\infty} \int_{Q(\tau_n,\rho_k)\setminus
Q(x_0,\rho_k)}(1+|\nabla u_n|)\, dx & \leq  \limsup_{r \to
1^+}\,\limsup_{n \to +\infty}\int_{Q(x_0,r\rho_k) \setminus
Q(x_0,\rho_k)}(1+|\nabla
u_n|)\, dx\\
 &\leq   \limsup_{r \to 1^+}\lambda \left( \overline{Q(x_0,r \rho_k) \setminus Q(x_0,\rho_k)}\right)\\
 &\leq  \lambda(\partial Q(x_0,\rho_k))=0\,,
\end{align*}
so that the last term in (\ref{lastterm})  vanishes. Hence
$$\mu(Q(x_0,\rho_k)) \geq  \limsup_{n \to +\infty} \,
\int_{Q(0,\rho_k)} f\left(\frac{x}{\e_n},\nabla v_n \right)dx\,,$$
where $\{v_n\} \subset W^{1,1}(Q(0,\rho_k);\M)$
satisfies $v_n \to u(x_0+\cdot)$ in $L^1(Q(0,\rho_k);\Rb^d)$  by (\ref{convv}).
\vskip5pt

Now we consider for every $n$, a sequence $\{v_{n,j}\}\subset {\cal
C}^\infty\big(\overline{Q(0,\rho_k)};\Rb^d\big)$ such that
$v_{n,j}\to v_n$ in $W^{1,1}(Q(0,\rho_k);\Rb^d)$, $v_{n,j}\to v_n$
and $\nabla v_{n,j}\to\nabla v_n$ a.e. in $Q(0,\rho_k)$ as
$j\to+\infty$ (we emphasize that in general, $v_{n,j}$ is not
$\M$-valued). Considering the integrand $g$ given by Lemma
\ref{defg}, one may check
\begin{multline*}
\lim_{j\to+\infty} \int_{Q(0,\rho_k)} g\left(\frac{x}{\e_n},v_{n,j},\nabla v_{n,j}
\right)dx= \int_{Q(0,\rho_k)} g\left(\frac{x}{\e_n},v_{n},\nabla v_{n}
\right)dx\,=\\
=\int_{Q(0,\rho_k)} f\left(\frac{x}{\e_n},\nabla v_n \right)dx\,,
\end{multline*}
so that we can find a diagonal sequence $\bar v_n:=v_{n,j_n}$ satisfying
$\bar v_n \to u(x_0+\cdot)$ in $L^1(Q(0,\rho_k);\Rb^d)$ and
\begin{equation}\label{imp}
\mu(Q(x_0,\rho_k)) \geq  \limsup_{n \to +\infty} \,
\int_{Q(0,\rho_k)} g\left(\frac{x}{\e_n},\bar v_{n},\nabla \bar v_{n}
\right)dx\,.
\end{equation}
Changing variable in
(\ref{imp}) yields
\begin{align}\label{CV}
\frac{d\mu}{d\LL^N}(x_0) & \geq \limsup_{k \to +\infty}\, \limsup_{n
\to +\infty} \,\int_Q g\left(\frac{\rho_k
\,x}{\e_n},\bar v_{n}(\rho_k\, x),  \nabla \bar v_{n}(\rho_k \, x)\right)dx\nonumber\\
& =  \limsup_{k \to +\infty} \limsup_{n \to +\infty}\, \int_Q
g\left(\frac{\rho_k \, x}{\e_n},s_0+\rho_k\, w_{n,k}, \nabla
w_{n,k}\right)dx\,,
\end{align}
where we have set $w_{n,k}(x):=\big[\bar v_{n}(\rho_k\, x) - s_0
\big]/\rho_k$. Since $x_0$ is a point of approximate
differentiability of $u$ and $\bar v_{n} \to u(x_0+\cdot)$ in
$L^1(Q(0,\rho_k);\Rb^d)$, we have
\begin{equation}\label{approxdiff}
\lim_{k \to +\infty} \lim_{n \to +\infty} \int_Q |w_{n,k}(x) -
\xi_0\, x|\, dx=  \lim_{k \to +\infty} \int_{Q(x_0,\rho_k)}
\frac{|u(y) - s_0 - \xi_0\,(y-x_0)|}{\rho_k^{N+1}}\, dy=0\,.
\end{equation}
In view of (\ref{CV}) and (\ref{approxdiff}), we can find a
diagonal sequence $\e_{n_k} < \rho_k^2$ such that $w_k:=w_{n_k,k}\to w_0$
in $L^1(Q;\Rb^d)$ with $w_0(x):=\xi_0\, x$, and
\begin{equation}\label{coerc}\frac{d\mu}{d\LL^N}(x_0) \geq \limsup_{k \to +\infty}\,
\int_Q g\left(\frac{x}{\d_k},s_0+\rho_k \,w_k,\nabla w_k \right)dx\,,
\end{equation}
where $\d_k:=\e_{n_k}/\rho_k\to0$.
\vskip5pt

{\bf Step 2.} We now argue as in Step 3 of the proof of
\cite[Theorem~2.1]{FM} to show that there exists a sequence
$\{\overline w_k\} \subset W^{1,\infty}(Q;\Rb^d)$ such that
$\overline w_k \to w_0$ in $L^\infty(Q;\Rb^d)$, $\{\overline w_k\}$
is uniformly bounded in $W^{1,1}(Q;\Rb^d)$ and
\begin{equation}\label{claim}
\frac{d\mu}{d\LL^N}(x_0) \geq \limsup_{k \to +\infty}\, \int_Q
g\left(\frac{x}{\d_k},s_0+\rho_k \,\overline  w_k, \nabla \overline
w_k\right)dx\,.
\end{equation}
Given $0<s<t\,$, let $\zeta_{s,t} \in
\C_c^\infty(\Rb;[0,1])$ be a cut-off function satisfying
$\zeta_{s,t}(\tau)=1$ if $|\tau| \leq s$, $\zeta_{s,t}(\tau)=0$ if
$|\tau| \geq t$ and $|\zeta'_{s,t}| \leq C/(t-s)$. Define
$$w_{s,t}^k:=w_0 + \zeta_{s,t}(|w_k - w_0|)(w_k - w_0)\,.$$
Obviously,
\begin{equation}\label{Linfty}
\|w_{s,t}^k - w_0 \|_{L^\infty(Q;\Rb^d)} \leq t\,,
\end{equation}
and the Chain Rule formula gives
\begin{equation}\label{derive}
\nabla w_{s,t}^k =  \nabla w_0 + \zeta_{s,t}\big(|w_k - w_0|\big)
(\nabla w_k - \nabla w_0) +  \zeta'_{s,t}\big(|w_k - w_0|\big) (w_k - w_0) \otimes \nabla
|w_k-w_0|\,.
\end{equation}
In particular,
\begin{multline}\label{huge}
\int_Q g\left(\frac{x}{\d_k},s_0+\rho_k \,w_{s,t}^k,\nabla
w_{s,t}^k\right)dx = \int_{\{|w_k-w_0| \leq s\}}
g\left(\frac{x}{\d_k},s_0+\rho_k \, w_k, \nabla
w_k\right)dx\,+\\
+\int_{\{s < |w_k-w_0| \leq t\}}
g\left(\frac{x}{\d_k},s_0+\rho_k \,w_{s,t}^k,\nabla
w_{s,t}^k\right)dx\,+\\
+\int_{\{|w_k-w_0| > t\}} g\left(\frac{x}{\d_k},s_0+\rho_k\, w_0,
\xi_0 \right)dx\,.
\end{multline}
>From the growth condition (\ref{pgrowth}), we infer that
\begin{equation}\label{1}
\int_{\{|w_k-w_0| > t\}}
g\left(\frac{x}{\d_k},s_0+\rho_k\, w_0 ,\xi_0\right)dx \leq
\b'(1+|\xi_0|) \, \LL^N(\{|w_k-w_0| > t\})\,,
\end{equation}
and (\ref{derive}) yields
\begin{multline}\label{2}
\int_{\{s < |w_k-w_0| \leq t\}} g\left(\frac{x}{\d_k},s_0+\rho_k
\,w_{s,t}^k,\nabla w_{s,t}^k\right)dx \leq C\int_{\{s < |w_k-w_0| \leq t\}}(1+|\nabla w_k -
\xi_0|)\, dx\,+\\
+\frac{C}{t-s}\int_{\{s < |w_k-w_0| \leq t\}}|w_k -
w_0|\, \big|\nabla |w_k - w_0|\,\big|\, dx\,.
\end{multline}
Observe that for $\LL^1$-a.e. $t>0$,
\begin{equation}\label{family}
\lim_{s \to t^-}\int_{\{s < |w_k-w_0| \leq t\}}(1+|\nabla w_k -
\xi_0|)\, dy \leq   C_k \, \lim_{s \to
t^-}\LL^N\big(\{s < |w_k-w_0| \leq t\}\big)=0\,,
\end{equation}
and by the Coarea formula,
\begin{multline}\label{coarea}
\lim_{s \to t^-} \frac{1}{t-s}\int_{\{s < |w_k-w_0| \leq t\}}|w_k
- w_0|\, \big|\nabla |w_k - w_0|\, \big|\, dx \,=\\
=  \lim_{s \to t^-}
\frac{1}{t-s} \int_s^t\tau \HH^{N-1}(\{|w_k-w_0|=\tau\})\, d\tau
=  t \HH^{N-1}(\{|w_k-w_0|=t\})\,.
\end{multline}
Moreover, in view of (\ref{pgrowth})
and (\ref{coerc}) we infer that
$$\int_Q\big|\nabla |w_k - w_0|\, \big|\, dx \leq
C\int_Q(1+|\nabla w_k|)\, dy \leq C_0\,.$$
Applying
\cite[Lemma~2.6]{FM}, there exists $t_k \in \left(
\|w_k-w_0\|^{1/2}_{L^1(Q;\Rb^d)},\|w_k-w_0\|^{1/3}_{L^1(Q;\Rb^d)}\right)$
such that (\ref{family}) and (\ref{coarea}) hold with $t=t_k$, and
\begin{equation}\label{log}t_k \HH^{N-1}(\{|w_k-w_0|=t_k\}) \leq
\frac{C_0}{\ln\left(\|w_k-w_0\|^{-1/6}_{L^1(Q;\Rb^d)}\right)}\,.
\end{equation}
According to (\ref{family}), (\ref{coarea})
and (\ref{log}), there exists $s_k \in  \left(
\|w_k-w_0\|^{1/2}_{L^1(Q;\Rb^d)},t_k\right)$ such that
\begin{equation}\label{coareabis}
\ds \int_{\{s_k < |w_k-w_0| \leq t_k\}}(1+|\nabla
w_k - \xi_0|)\, dx\leq \frac{1}{k}\,,
\end{equation}
and
\begin{equation}\label{coareabisbisbis}
\frac{1}{t_k-s_k}\int_{\{s_k < |w_k-w_0| \leq t_k\}}|w_k -w_0|\,
\big|\nabla |w_k - w_0|\, \big|\; dx \leq
\frac{C_0}{\ln\left(\|w_k-w_0\|^{-1/6}_{L^1(Q;\Rb^d)}\right)}+\frac{1}{k}\,,
\end{equation}
while (\ref{1}) together with Chebyshev inequality yields
\begin{equation}\label{1bis}\int_{\{|w_k-w_0| > t_k\}}
g\left(\frac{x}{\d_k},s_0+\rho_k\, w_0, \xi_0 \right)dy \leq C
\|w_k-w_0\|^{1/2}_{L^1(Q;\Rb^d)}\,.
\end{equation}
Define now $\overline  w_k:=w_{s_k,t_k}^k$ so that $\overline  w_k \to w_0$ in $L^\infty(Q;\Rb^d)$ by (\ref{Linfty}).
Moreover, gathering (\ref{huge}), (\ref{2}),
(\ref{coareabis}), (\ref{coareabisbisbis}) and (\ref{1bis}), we deduce
\begin{equation*}
\limsup_{k \to +\infty}\int_Q g\left(\frac{x}{\d_k},s_0+\rho_k
\,\overline  w_k,\nabla \overline  w_k\right)dx
\leq \limsup_{k \to +\infty}\int_Q
g\left(\frac{x}{\d_k},s_0+\rho_k \,w_k,\nabla w_k\right)dx\,,
\end{equation*}
which proves (\ref{claim}). The fact that $\{\nabla \overline w_k\}$
is uniformly bounded in $L^1(Q;\Rb^{d \times N})$ is a consequence
of
(\ref{claim}) and the coercivity condition (\ref{pgrowth}).
\vskip5pt

{\bf Step 3.} Since $\{\|\overline
w_k\|_{L^\infty(Q;\Rb^d)}\}$ and $\{\|\nabla \overline
w_k\|_{L^1(Q;\Rb^{d \times N})}\}$ are uniformly bounded,  we derive from
(\ref{moduluscont}) that
$$\lim_{k \to +\infty}\int_Q
\left|\, g\left(\frac{x}{\d_k},s_0+\rho_k \, \overline  w_k, \nabla
\overline w_k\right) - g\left(\frac{x}{\d_k},s_0, \nabla \overline
w_k\right)\right| dx=0\,.$$
In view of (\ref{claim}), it leads to
$$\frac{d\mu}{d\LL^N}(x_0) \geq \lim_{k \to +\infty}
\int_Q g\left(\frac{x}{\d_k},s_0, \nabla \overline  w_k\right)dx\,.$$
Using standard homogenization results (see {\it e.g., }
\cite[Theorem~14.5]{BD}) together with (\ref{f=g}), we finally conclude that
$$\frac{d\mu}{d\LL^N}(x_0)  \geq g_\hom(s_0,\xi_0)=Tf_\hom(s_0,\xi_0)\,,$$
which completes the proof of the lemma.
\end{proof}

\subsection{Proof of Theorem \ref{babmil} and Theorem \ref{babmilp=1} completed}

Since $L^p(\O;\Rb^d)$ is separable ($1\leq p <+\infty$),  there
exists a subsequence  $\{\e_{n_k}\}$ such that $\F$ is the
$\G$-limit of $\{\F_{\e_{n_k}}\}$ for the strong
$L^p(\O;\Rb^d)$-topology (see \cite[Theorem~8.5]{DM}).
\vskip 5pt

\noindent {\it Case $p>1$. } In view of $(H_2)$ and the closure of
the pointwise constraint under strong $L^p$-convergence, we  have
$\F(u)<+\infty$ if and only if $u \in W^{1,p}(\O;\M)$. Hence, as a
consequence of Lemmas \ref{upperbound} and \ref{lowerbound},
the functionals $\{\F_{\e_{n_k}}\}$ $\G$-converge to $\F_\hom$ in $L^p(\O;\Rb^d)$.
Since the $\G$-limit does not depend on the
extracted subsequence, we get in light of \cite[Proposition~8.3]{DM}
that the whole sequence $\{\F_{\e_n}\}$ $\G$-converges to $\F_{\rm hom}$.
\vskip5pt

\noindent {\it Case $p=1$. }As a consequence
of Lemmas \ref{upperbound}
and \ref{lowerbound1},
the functionals $\{\F_{\e_{n_k}}\}$ $\G$-converge to $\F_\hom$ in $W^{1,1}(\O;\M)$.
Again, the $\G$-limit does not depend on the
extracted subsequence, so that the whole sequence $\{\F_{\e_n}\}$ $\G$-converges
to $\F_{\rm hom}$ in $W^{1,1}(\O;\M)$. \prbox

\vskip15pt

\noindent{\bf Acknowledgement. }The authors wish to thank Roberto
Alicandro, Pierre Bousquet, Giovanni Leoni and Domenico Mucci for
several interesting discussions on the subject. This work was initiated while
V. Millot was visiting the department of {\it Functional Analysis and Applications} at S.I.S.S.A.,
he thanks G.~Dal~Maso  and the whole department for
the warm hospitality. The research of
J.-F. Babadjian was partially supported by the Marie Curie Research
Training Network MRTN-CT-2004-505226 ``Multi-scale modelling and
characterisation for phase transformations in advanced materials''
(MULTIMAT). V. Millot was partially supported by  the Center for
Nonlinear Analysis (CNA) under the National Science Fundation Grant
No. 0405343.

\end{document}